\documentclass[11p]{article}
\date{}

\usepackage[noBBpl,sc,osf]{mathpazo} % math & rm
\usepackage[scaled]{helvet} % ss
\usepackage{zlmtt} % tt
\usepackage[T1]{fontenc}
\usepackage{textcomp}
\normalfont
\usepackage{amssymb}

\usepackage{times}
\usepackage{amsthm}
\usepackage{amsmath}
\usepackage{amssymb}
\usepackage{mathrsfs}
\usepackage{pb-diagram}
\usepackage{xcolor}

\usepackage{hyperref}

\newcommand{\A}{\mathcal{A}}

\newcommand{\h}{\mathcal{H}}

\newcommand{\rtt}{\rightthreetimes}

\usepackage{tikz-cd}

\def\o{\omega}
\def\O{\Omega}
\def\te{\theta}
\def\Te{\Theta}

\def\T{\mathbb{T}}
\def\Z{\mathbb{Z}}

\def\h{\mathcal H}

\def\Z{\mathbb Z}

\def\e{{\sf e}}

\def\d{{\rm d}}
\def\bu{\bullet}

\def\({\left(}
\def\[{\left[}
\def\){\right)}
\def\]{\right]}

\def\si{\sigma}
\def\Si{\Sigma}
\def\G{{\sf G}}

\def\p{\parallel}
\def\<{\langle}
\def\>{\rangle}
\providecommand{\norm}[1]{\lVert#1\rVert}

 \newtheorem{thm}{Theorem}[section]
 \newtheorem{cor}[thm]{Corollary}
 
 \newtheorem{lem}[thm]{Lemma}
 \newtheorem{prop}[thm]{Proposition}
 \theoremstyle{definition}
 \newtheorem{defn}[thm]{Definition}
 \theoremstyle{remark}
 \newtheorem{rem}[thm]{Remark}
 \newtheorem{ex}[thm]{Example}
 \numberwithin{equation}{section}

\numberwithin{equation}{section}

\begin{document}

%-------------------------------------------------------------------------------------------------------
% Title
%-------------------------------------------------------------------------------------------------------

\title{Symmetry for algebras associated to Fell bundles over groups and groupoids}

\author{Felipe Flores, Diego Jaur\'e and Marius M\u antoiu
}

%\address{Facultad de Ciencias, Departamento de Matem\'aticas\\Universidad de Chile\\Las Palmeras 3425, Casilla 653,\\Santiago, Chile.}

%\email{mantoiu@uchile.cl}

\maketitle

%-------------------------------------------------------------------------------------------------------
 %Abstract
%-------------------------------------------------------------------------------------------------------

\begin{abstract}
To every Fell bundle $\mathscr C$ over a locally compact group $\G$ one associates a Banach $^*$-algebra $L^1(G\,\vert\,\mathscr C)$\,. We prove that it is symmetric whenever $\G$ with the discrete topology is rigidly symmetric. This generalizes the known case of a global action without a twist. There is also a weighted version as well as a treatment of some classes of associated integral kernels. We also deal with the case of Fell bundles over discrete groupoids. We formulate a generalization of rigid symmetry in this case and show its equivalence with an a priori stronger concept. We also study the symmetry of transformation groupoids and some permanence properties.
\end{abstract}

%-------------------------------------------------------------------------------------------------------
\section{Introduction}\label{introduction}
%-------------------------------------------------------------------------------------------------------

This article treats the symmetry of certain Banach $^*$-algebras connected with Fell bundles over discrete groupoids and locally compact groups (always supposed to be Hausdorff). Let us recall the main classical concept: 

\begin{defn} \label{symmetric}
A Banach $^*$-algebra $\mathfrak B$ is called {\it symmetric} if the spectrum of $b^*b$ is positive for every $b\in\mathfrak B$ (this happens if and only if the spectrum of any self-adjoint element is real.)
\end{defn}

The symmetry of a Banach $^*$-algebra admits many reformulations and has many useful and interesting consequences \cite{Pa2,Gr}, which will not be discussed here. Let us only mention that, in a suitable framework, it is connected with notions as spectral invariance and stability under the holomorphic functional calculus. We also feel that this is not the right place to sketch a history of the subject. A very useful and readable presentation may be found in \cite{Gr}. More general $^*$-algebras are also studied from the point of view of symmetry, but this will no longer be mentioned here.

\smallskip
A great deal of effort has been dedicated to Banach algebras associated to a locally compact group $\G$\,. The basic example is the convolution algebra $L^1(\G)$\,; actually the interest in the symmetry property arouse around the Banach algebra interpretation and treatment of the classical result \cite{Wie} of Wiener on Fourier series. But there are increasingly general other classes, as global crossed products, partial crossed products (both twisted by $2$-cocycles or not), groupoid algebras and $L^1$-types algebras associated to Fell bundles. All of these played an important theoretical role and lead to many applications. When looking for results and examples, one aims to enlarge the collection of groups of groupoids that can be treated, as well as the class of symmetric Banach $^*$-algebras assigned to them. {\it The present paper is concerned with these purposes, adopting the very general point of view of Fell bundles.}

\smallskip
Until very recently, the largest class for which general results have been found was that of crossed products attached to a global action of the group over a $C^*$-algebra. A cohomological twist has also been permitted. The simplest case of a trivial action leads to the projective tensor product between $L^1(\G)$ and a $C^*$-algebra. Some references are: \cite{LP,Bas,Kur,Po,Pa2,GL1,GL2,Ba,FGL,Gr,BB,Ma,FL}. Groupoid algebras are much less studied in the setup of symmetry and spectral invariance; a project on this topic started in \cite{AO}.

\smallskip
The main body of the article is composed of two parts: Section \ref{maininsultat} deals with Fell bundles over discrete groupoids, while Section \ref{flaconir} treats Fell bundles over locally compact groups. These are significant extensions in different directions of a result from \cite{JM}. Although connected by terminology and some of the techniques, the reader could read Section \ref{maininsultat} and Section \ref{flaconir} (and \cite{JM})  separately. We provide a description of the two sections. But first it is convenient to use the following terminology; the points (i) are (ii) are classical notions. (An adaptation for groupoids  will follow below.)

\begin{defn} \label{rigidlysymmetric}
\begin{enumerate}
\item[(i)]
The locally compact group $\G$ is called {\it symmetric} if the convolution Banach $^*$-algebra $L^1(\G)$ is symmetric.
\item[(ii)]
The locally compact group $\G$ is called {\it rigidly symmetric} if given any $C^*$-algebra $\A$\,, the projective tensor product $L^1(\G)\otimes\A$ is symmetric.
\item[(iii)]
The locally compact group $\G$ is called {\it hypersymmetric} if for every Fell bundle $\mathscr C$ over $\G$\,, the Banach $^*$-algebra $L^1(\G\,\vert\,\mathscr C)$ is symmetric.
\end{enumerate}
\end{defn}

The main result of \cite{JM}, stated in the different but equivalent language of graded algebras, says that {\it rigid symmetry and} (the seemingly much stronger) {\it hypersymmetry are equivalent for discrete groups}. This was supported by many interesting examples in which the group and the Fell bundle structure were not even visible.

%-------------------------------------------------------------------------------------------------------
\subsection{The case of discrete groupoids}\label{maininsult}
%-------------------------------------------------------------------------------------------------------

This section aims mainly at extending the mentioned result from \cite{JM} to Fell bundles over discrete groupoids. In this case, even the concept of rigid symmetry must be suitably adapted. Unfortunately, the symmetry issue in a groupoid setting is in a very incipient form (see \cite{AO}, however), so classes of rigidly symmetric groupoids still have to be found, which will then insure the natural and general hypersymmetry, by Theorem \ref{principala}.
As far as we know, the idea to embed more complicated algebras attached to discrete groups into simpler ones (projective tensor products of the usual $L^1$ group algebra with $C^*$-algebras) originated in \cite{FGL}, and it has been used in other papers. Here we adapt it to the framework of Fell bundles over discrete groupoids, which seems to be the most general object for which such technique is viable. Anyhow, this covers directly many natural examples, as transformation groupoids assigned to discrete group actions on discrete spaces and various pull-backs and equivalence relations in a discrete setting. Theorem \ref{teoremix} makes available the result for discrete groups in a form that will be useful in the next section.  

\smallskip
Going beyond the discrete groupoids (of a general form) seems difficult, since in this case the integration/disintegration theory for Fell bundle representations becomes very intricate, cf. sections 3,4,5 in \cite{MW}. This is why an analog of Theorem \ref{theoremix} for locally compact groupoids is still out of reach, and in Section \ref{flaconir} we restrict to groups.

\smallskip
In subsection \ref{mitanik} we provide an application to actions of discrete groupoids on $C^*$-bundles. In the Abelian case, one gets a result about the symmetry of transformation groupoids (the action space is no longer discrete).

\smallskip
Finally, in subsection \ref{minimor} we study morphisms.  In Theorem \ref{acidboroch} and Corollary \ref{acidboricoci} it is shown that subgroupoids of symmetric groupoids are also symmetric (and similarly for hypersymmetry). This should be compared with Proposition 4.2 from \cite{AO}. Theorem \ref{acidborich} shows that symmetry transfers from a discrete groupoid to any of its epimorphic image.

%-------------------------------------------------------------------------------------------------------
\subsection{The case of locally compact groups}\label{flaconrir}
%-------------------------------------------------------------------------------------------------------

The initial motivation of \cite{JM} and of the present article was the fact that partial group actions are very general and that many of the interesting examples are not coming from a global action. It soon became clear that Fell bundles, which are even more general, can also be treated in a similar way. In \cite{JM} the case of a discrete group was considered. Here, using ideas of Poguntke \cite{Po}, we treat Fell bundles over a locally compact group $\G$\,, but for technical reasons we need to impose conditions on the same group $\G^{\rm dis}$ with the discrete topology.

\smallskip
A diagram with implications and equivalences is supposed to systematize the actual state of art. If $\G$ is a locally compact group, we will denote by $\G^{\rm dis}$ the same group with the discrete topology. 
%$$\begin{diagram}
%\node{\G\ {\rm symmetric}}\node{\G\ {\rm rigidly\ symmetric}}\arrow{w,t}{2}\node{\G\ {\rm hypersymmetric}}\arrow{w,t}{7}\\ 
%\node{\G^{\rm dis}\ {\rm symmetric}}\arrow{n,l}{3}\node{\G^{\rm dis}\ {\rm rigidly\ symmetric}}\arrow{n,l}{4}\arrow{ne,t}{6}\arrow{e}\arrow{w,t}{1}\arrow{e,t}{5}\node{\G^{\rm dis}\ {\rm %hypersymmetric}}\arrow{w}\arrow{n,l}{8}
%\end{diagram}$$
$$
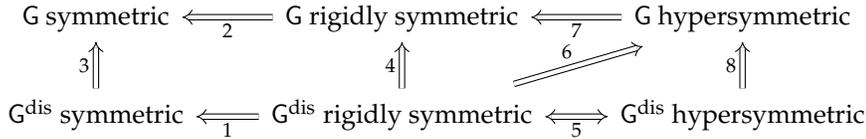
\begin{tikzcd}
[arrows=Rightarrow]
    \G\ {\rm symmetric}
    & \G\ {\rm rigidly\ symmetric}
    \arrow{l}{I_2}
    & \G\ {\rm hypersymmetric}
    \arrow{l}{I_7}
    \\
    \G^{\rm dis}\ {\rm symmetric}
    \arrow{u}{I_3}
    & \G^{\rm dis}\ {\rm rigidly\ symmetric} 
    \arrow{l}{I_1}
    \arrow{u}{I_4}
    \arrow[shorten <= 6pt,shorten >= 6pt]{ur}{I_6}
    & \G^{\rm dis}\ {\rm hypersymmetric}
    \arrow[Leftrightarrow]{l}{I_5}
    \arrow{u}{I_8}
\end{tikzcd}
$$

Clearly, a rigidly symmetric group is symmetric, which is expressed by the two horizontal arrows $I_1$ and $I_2$\,. It is still not known if the two notions are really different.

\smallskip
The implications $I_3$ and $I_4$ are due to Poguntke \cite{Po}. 

\smallskip
In $I_5$ one direction is trivial, since projective tensor products of the form $L^1(\G)\otimes\A$ are easily written as $L^1$-algebras of some Fell bundle (a similar statement holds for $I_7$)\,. The fact that, for discrete groups, rigid symmetry implies hypersymmetry is one of the main results of \cite{JM}.

\smallskip
{\it The main results of Section \ref{flaconir} are the (equivalent) implications $I_6$ and $I_8$.}

\smallskip
Replacing the implication $I_7$ by an equivalence would certainly be considered a nice result. Neither arrow $I_3$ nor arrow $I_4$ can be reversed: Any compact connected semisimple real Lie group is rigidly symmetric, but in \cite{BG} it is shown that it contains a (dense) free group on two elements. The discretization will not be amenable; by \cite{SW} it cannot be symmetric. 

\smallskip
In \cite{FL} there is a global action of the 'ax+b' group, whose associated $L^1$-algebra is not symmetric. This provides an example of a symmetric but not hypersymmetric group. But this does not say which of the implications $I_2$ or $I_7$ cannot be reverted.

\smallskip
Since in $I_6$ one starts with rigid symmetry of the discretization, it is relevant to have examples. Classes of rigidly symmetric discrete groups are (cf. \cite{LP}): (a) Abelian, (b) finite, (c) finite extensions of discrete nilpotent. This last class includes all the finitely generated groups with polynomial growth. A central extension of a rigidly symmetric group is rigidly symmetric, by \cite[Thm.\,7]{LP}. In \cite[Cor.\,2.16]{Ma} it is shown that the quotient of a discrete rigidly symmetric group by a normal subgroup is rigidly symmetric.

\smallskip
In Section \ref{flocinor} we recall briefly the terminology of Fell bundles and state the main results. The implication $I_6$ is formulated in Theorem \ref{theoremix}. As a consequence of \cite[Thm.\,1]{Ku}, if $\G$ is supposed only symmetric, but the $C^*$-algebra $\A$ is type $I$, the projective tensor product $\ell^1(\G)\otimes\A$ is symmetric. This allows us to state and prove Theorem \ref{theoremix1} which cannot be obtained directly from the implication $I_6$. The passage from $I_5$ to $I_6$ is stated in Theorem \ref{cruciala}, which is the main technical result.

\smallskip
The proof of this crucial Theorem \ref{cruciala} is done in Sections \ref{pispis} and \ref{proof}, following ideas of Poguntke \cite{Po} (see also \cite{BB1,FL}). It makes use of a criterion for symmetry of Banach $^*$-algebras in terms of representations and a discretization procedure. The fact that the $L^1$-sections of a Fell bundle do not take values in a single Banach space is a source of complications. 

\smallskip
A {\it twisted partial action} $(\theta,\o)$ of a locally compact group $\G$ by partial isomorphisms between ideals of a $C^*$-algebra $\A$ gives rise to a {\it twisted crossed product} $\A\rtt_\te^\o\G$ and to a Fell bundle \cite{Ex1,Ab}. It is even true that if a Fell bundle is separable and the unit fibre is stable, there is always a twisted partial action $(\G,\theta,\o,\A)$ around. In Section \ref{goam} we sketch the constructions and write down the symmetry result in this case, which contains other important families (global actions, trivial cocycles).

\smallskip
In Section \ref{flocinyor} we present a symmetry result involving weights.

\smallskip
In connection with a given Fell bundle, in Section \ref{pishpish} we introduce a class of kernel-sections that are convolution dominated with respect to a weight. Those which are covariant, in a suitable way, form a symmetric Banach $^*$-algebra. To get a better result, for a larger algebra, we define an enlarged Fell bundle, canonically associated to the initial one.

%-------------------------------------------------------------------------------------------------------
\section{The case of discrete groupoids}\label{maininsultat}
%-------------------------------------------------------------------------------------------------------

%-------------------------------------------------------------------------------------------------------
\subsection{Fell bundles over discrete groupoids and symmetry issues}\label{mainresult}
%-------------------------------------------------------------------------------------------------------

Let $\Xi$ be a discrete groupoid, with unit spce $\Xi^{(0)}\!=:\!U$\,, source map ${\rm s}$ and range map ${\rm r}$\,. The set of composable pairs is 
$$\Xi^{(2)}\!:=\{(x,y)\!\mid\!{\rm r}(y)={\rm s}(x)\}\,.
$$ 

We are going to work with {\it Fell bundles} $\mathscr C\overset{q}{\to}\Xi$ over the groupoid ${\Xi}$ (see \cite {Yam,Kum,MW,TWG}). Without spelling out the whole definition in detail, let us recall that each fibre $\mathfrak C _x\!:=q^{-1}(\{x\})$ is a Banach space with norm $\p\!\cdot\!\p_{\mathfrak C_x}$\,, the topology of $\mathscr C$ coincides with the norm topology on each fibre, there are antilinear continuous involutions 
$$
\mathfrak C _x\ni \!a\to a^{\bu}\!\in\mathfrak C _{x^{-1}}
$$ 
and for all $(x,y)\in\Xi^{(2)}$ there are continuous multiplications 
$$
\mathfrak C _x\!\times\!\mathfrak C _y\ni(a,b)\to a\bu b\in\mathfrak C _{xy}
$$ 
satisfying the following axioms for $a\in\mathfrak C_x\,,b\in\mathfrak C_y\,,(x,y)\in\Xi^{(2)}$\,: 
\begin{itemize}
\item $\p\!ab\!\p_{\mathfrak C_{xy}}\,\le\,\p\!a\!\p_{\mathfrak C_{x}}\p\!b\!\p_{\mathfrak C_{y}}$\,,
\item $(ab)^\bu=b^\bu a^\bu$,
\item $\p\!a^\bu a\!\p_{\mathfrak C_{{\rm s}(x)}}=\,\p\!a\!\p_{\mathfrak C_{x}}^2$\,,
 \item $a^\bu a$ is positive in $\mathfrak C_{{\rm s}(x)}$\,.
 \end{itemize}
From these axioms it follows that $\mathfrak C _u$ is a $C^*$-algebra for every unit $u\in U$\,. Sometimes we simply write $\mathscr C=\bigsqcup_{x\in\Xi}\mathfrak C _x$ for the Fell bundle. 

\smallskip
The most important object for us is {\it the Hahn algebra} $\ell^{\infty,1}(\Xi\!\mid\!\mathscr C)$ adapted to Fell bundles  \cite{MW}, which in our discrete case if formed by the sections $\Phi:\Xi\to\mathfrak C$ (thus satisfying $\Phi(x)\in\mathfrak C_x$ for every $x\in\Xi$) such that {\it the Hahn-type norm} 
\begin{equation}\label{parts}
\p\!\Phi\p_{\infty,1}\,:=\max\Big\{\sup_{u\in U}\sum_{{\rm r}(x)=u}\!\p\!\Phi(x)\!\p_{\mathfrak C_x}\,,\,\sup_{u\in U}\sum_{{\rm s}(x)=u}\!\p\!\Phi(x)\!\p_{\mathfrak C_x}\!\Big\}
\end{equation}
is finite. It is a Banach $^*$-algebra under the multiplication
\begin{equation*}\label{tiplication}
(\Phi\star \Psi)(x):=\sum_{yz=x}\Phi(y)\bullet\Psi\big(z)
\end{equation*}
and the involution
\begin{equation*}\label{tion}
\Phi^\star(x):=\Phi\big(x^{-1}\big)^\bu.
\end{equation*}
The space $C_{\rm c}(\Xi\!\mid\!\mathscr C)$ of finitely-supported sections forms a dense $^*$-algebra of the Hahn algebra. The complexity of the multiplication, largely responsable for the generality of the emerging algebras, comes both from the complexity of the 'inner' Fell multiplication $\bu$ and from the groupoid-type convolution inherent to the formula. 

\smallskip
We need some special Fell bundles associated to Hilbert bundles $\mathscr H\!:=\bigsqcup_{u\in U}\h_u$ over the unit space; here the fact that $\Xi$ is discrete will be crucial. For $u,v\in U$ we set $\mathbb B(\h_u,\h_v)\equiv\mathbb B(u,v)$ for the Banach space of all bounded linear operators $A:\h_u\to\h_v$\,. Taking advantage of the norm, the obvious  multiplications $$
\mathbb B(\h_w,\h_v)\!\times\!\mathbb B(\h_u,\h_w)\to\mathbb B(\h_u,\h_v)
$$ 
and the obvious  involutions $\mathbb B(\h_u,\h_v)\to\mathbb B(\h_v,\h_u)$\,, one constructs the Fell bundle 
$$
\mathbb B^\mathscr H\!:=\!\!\bigsqcup_{(u,v)\in U\times U}\!\mathbb B(u,v)\to U\!\times\!U
$$ 
over the pair groupoid. Actually we are interested in the pull-back Fell bundle 
$$
\mathscr B^\mathscr H\!:=({\rm s},{\rm r})^*\big(\mathbb B^\mathscr H\big)
$$ 
of $\mathbb B$  through the groupoid morphism $({\rm s},{\rm r}):\Xi\to U\!\times\!U$. We denote it by 
$$
\mathscr B^\mathscr H\!:=\!\bigsqcup_{x\in\Xi}\mathfrak B_x\overset{p}{\to}\Xi\,,
$$ 
with fibres $\mathfrak B_x\!:=\mathbb B\big({\rm s}(x),{\rm r}(x)\big)$ and obvious structure.

\begin{defn} \label{groupsymmetric}
\begin{enumerate}
\item[(i)]
The discrete groupoid $\Xi$ is called {\it symmetric} if the convolution Banach $^*$-algebra $\ell^{\infty,1}(\Xi)$ is symmetric.
\item[(ii)]
The discrete groupoid $\Xi$ is called {\it rigidly symmetric} if, given any Hilbert bundle $\mathscr H\!=\bigsqcup_{x\in X}\h_x$ over its unit space, the Banach $^*$-algebra $\ell^{\infty,1}\big(\Xi\,\big\vert\,\mathscr B^\mathscr H\big)$ is symmetric.
\item[(iii)]
The discrete groupoid $\Xi$ is called {\it hypersymmetric} if given any Fell bundle $\mathscr C\!=\bigsqcup_{x\in\Xi}\mathfrak C _x$\,, the Banach $^*$-algebra $\ell^{\infty,1}(\Xi\!\mid\!\mathscr C)$ is symmetric.
\end{enumerate}
\end{defn}

It is clear that rigid symmetry implies symmetry; just take the constant field $\h_u\!:=\mathbb C$ for every $u$\,. As said in the Introduction, even for discrete groups it is still unknown if the two notions are identical.

\begin{thm}\label{principala}
For a discrete groupoid, rigid symmetry and hypersymmetry coincide.
\end{thm}

\begin{proof}
The fact that hypersymmetry implies rigid symmetry is obvious, since $\mathscr B^\mathscr H$ is a particular type of discrete groupoid Fell bundle. So we need to show the converse implication. 

\smallskip
We will first prove a result concerning the existence of isometric representations of such Fell bundles. We rely on the integrated form of such a representation. In general the connections between representations and their integrated forms is an intricate issue (see \cite[Sect. 3,4,5]{MW} for instance), especially at the level of `disintegration`, but for discrete $\Xi$ this simplifies a lot. We will sketch the constructions we need without saying how they fit the general case. But just a hint: the counting measure on $U$ is invariant with respect to the Haar system composed of counting measures on the fibers of the discrete groupoid. 

Let $\mathscr C=\bigsqcup_{x\in\Xi}\mathfrak C _x\overset{q}{\to}\Xi$\, be an arbitrary Fell bundle and
$$
\pi:\mathscr C\to\mathscr B^\mathscr H\!=\bigsqcup_{x\in\Xi}\mathbb B\big(\h_{{\rm s}(x)},\h_{{\rm r}(x)}\big)
$$ 
be a representation, where $\mathscr H=\bigsqcup_{u\in U}\h _u$ is a Hilbert bundle over $U$\,. In our context this just means that $\pi$ is a morphism of Fell bundles. Its integrated form 
$$
\Pi:\ell^{\infty,1}(\Xi\!\mid\!\mathscr C)\to\mathbb B\Big(\bigoplus_{u\in U}\h_u\big)
$$
is defined by
\begin{equation}\label{surpr}
\big[\Pi(\Phi)h\big](u):=\!\sum_{{\rm r}(x)=u}\!\pi\big[\Phi(x)\big]h\big[{\rm s}(x)\big]\,,\quad\forall\,h\in\bigoplus_{u\in U}\h_u\,.
\end{equation}
For any unit $v$ we embed $\mathfrak C_v$ into $C_{\rm c}(\Xi\!\mid\!\mathscr C)\subset\ell^{\infty,1}(\Xi\!\mid\!\mathscr C)$\,, setting for each $a\in\mathfrak C_v$
$$
(\te_v a)(x):=a\ \ {\rm if}\ \ x=v\,,\quad(\te_v a)(x):=0_{\mathfrak C_x}\ \ {\rm if}\ \ x\ne v\,,
$$
and we get from \eqref{surpr}
$$
\big[\Pi(\te_v a)h\big](u)=\pi(a)h(v)\ \ {\rm if}\ \ u=v\,,\quad\big[\Pi(\te_v a)h\big](u)=0_{\h_u}\ \ {\rm if}\ \ u\ne v\,.
$$
It follows immediately that if $\Pi$ is injective, then the restriction of $\pi$ to $\mathfrak C_v$ is also injective, i.\,e. isometric. Actually $\Pi$ extends to the full $C^*$-algebra $C^*(\Xi\!\mid\!\mathscr C)$ of the Fell bundle, which contains densely $\ell^{\infty}(\Xi\!\mid\!\mathscr C)$\,. Injective representations of $C^*$-algebras do exist; we conclude that the Fell bundle representation $\pi$ is isometric on the $C^*$-algebras corresponding to the units. But then, by a standard argument, the isometry also propagates on all the Banach spaces of the Fell bundle: if $b\in\mathfrak C_x$ the axioms allow us to write
$$
\begin{aligned}
\p\!\pi(b)\!\p_{\mathfrak B_x}^2&=\,\p\!\pi(b)^*\pi(b)\!\p_{\mathfrak B_{{\rm s}(x)}}\\
&=\,\p\!\pi\big(b^\bu\bu b\big)\!\p_{\mathfrak B_{{\rm s}(x)}}\\
&=\,\p\!b^\bu\bu b\!\p_{\mathfrak C_{{\rm s}(x)}}\\
&=\,\p\!b\!\p_{\mathfrak C_x}^2.
\end{aligned}
$$
Now define 
$$
\Upsilon_\pi\!:\ell^{\infty,1}(\Xi\!\mid\!\mathscr C)\to\ell^{\infty,1}\big(\Xi\,\big\vert\,\mathscr B^\mathscr H\big)\,,\quad\big(\Upsilon_\pi(\Phi)\big)(x)\!:=\pi\big(\Phi(x)\big)\,.
$$
It is a well-defined linear isometry: we compute for the range part of the norm \eqref{parts} (using more detailed notations); the same is true for the source part.
$$
\begin{aligned}
\p\!\Upsilon_\pi(\Phi)\!\p^{{\rm r}}_{\ell^{\infty,1}(\Xi\,\vert\,\mathscr B^\mathscr H)}\,&=\sup_u\!\sum_{{\rm r}(x)=u}\!\p\!\pi\big(\Phi(x)\big)\!\p_{\mathfrak B_x}\\
&=\sup_u\!\sum_{{\rm r}(x)=u}\!\p\!\Phi(x)\!\p_{\mathfrak C_x}=\,\p\!\Phi\!\p^{{\rm r}}_{\ell^{\infty,1}(\Xi\mid\mathscr C)}.
\end{aligned}
$$
It is also an involutive morphism. For the multiplication, for instance, we have:
$$
\begin{aligned}
\big[\Upsilon_\pi(\Phi)\star\Upsilon_\pi(\Psi)\big](\Xi)&=\sum_{yz=x}\big[\Upsilon_\pi(\Phi)\big](y)\big[\Upsilon_\pi(\Psi)\big](z)\\
&=\sum_{yz=x}\pi[\Phi(y)]\pi[\Psi(z)]\\
&=\pi\Big(\sum_{yz=x}\Phi(y)\bu\Psi(z)\Big)\\
&=\big[\pi(\Phi\star\Psi)\big](x)\\
&=\big[\Upsilon_\pi(\Phi\star\Psi)\big](x)\,.
\end{aligned}
$$

Thus we proved that $\ell^{\infty,1}(\Xi\!\mid\!\mathscr C)$ can be embedded as a closed $^*$-algebra of the symmetric Banach $^*$-algebra $\ell^{\infty,1}\big(\Xi\,\big\vert\,\mathscr B^\mathscr H\big)$\,, so it is also symmetric, by \cite[Th.11.4.2]{Pa2}. The proof is finished.
\end{proof}

The next result, valid for discrete groups, has essentially been stated and proven in \cite{JM} in the language of graded algebras. Rephrasing it in the context of Fell bundles is straightforward. We re-obtain here the (most important) point (i) as a consequence of Theorem \ref{principala}.  It will be useful in the next sections.

\begin{thm}\label{teoremix}
Let $\mathscr D$ be a Fell bundle over the discrete group ${\sf H}$\,. Assume any of the hypothesis:
\begin{enumerate}
\item[(i)] ${\sf H}$ is rigidly symmetric (cf. Definition \ref{rigidlysymmetric}(ii)),
\item[(ii)] ${\sf H}$ is symmetric and $C^*\big({\sf H}\,\vert\,\mathscr D\big)$ is a type $I$ $C^*$-algebra.
\end{enumerate}
Then $\ell^1({\sf H}\,\vert\,\mathscr D)$ is a symmetric Banach $^*$-algebra. 
\end{thm}

\begin{proof}
(i)  If $\Xi\equiv{\sf H}$ is a discrete group with unit $\e$\,, since $U=\{\e\}$\,, the $\ell^{\infty,1}$-algebras reduce to the usual $\ell^1$-algebras associated to Fell bundles with discrete groups \cite{FD2}. The Hilbert bundle reduces to a single Hilbert space $\h_{\e}=:\h$\,, the Fell bundle $\mathscr B^\mathscr H$ is only composed of $\mathbb B(\h)$\, and then $\ell^{\infty,1}\big({\sf H}\,\big\vert\,\mathscr B^\mathscr H\big)$ becomes $\ell^1\big({\sf H},\mathbb B(\h)\big)$\,, which is isomorphic to the projective tensor product $\ell^1({\sf H})\!\otimes\!\mathbb B(\h)$\,. We recover in this case the classical notion of rigidly symmetric group, cf. \cite{Po} and Definition \ref{rigidlysymmetric}(ii). This explains a posteriori our terminology. In the same way 
$$
\ell^{\infty,1}({\sf H}\,\vert\,\mathscr D)=\ell^1({\sf H}\,\vert\,\mathscr D)\,,
$$ 
and then we apply Theorem \ref{principala}.
\end{proof}

%-------------------------------------------------------------------------------------------------------
\subsection{Actions of discrete groupoids on continuous fields of $C^*$-algebras}\label{mitanik}
%-------------------------------------------------------------------------------------------------------

\begin{defn}\label{groupoidaktion}
By {\it a left groupoid action} of our discrete groupoid $\Xi$ on the $C^*$-bundle $\mathscr A\!:=\bigsqcup_{u\in U}\mathfrak A_u\overset{p}{\to}U$ over its unit space  we understand
a continuous map 
$$
\mathscr A\rtimes\Xi:=\{(\alpha,x)\in\mathscr A\!\times\!\Xi\!\mid\!p(\alpha)={\rm s}(x)\}\ni(\alpha,x)\to \mathcal T(\alpha,x)\equiv\mathcal T_x(\alpha)\in\mathscr A,
$$ 
satifying the axioms
\begin{enumerate}
\item[(a)] $p\big[\mathcal T_x(\alpha)\big]={\rm r}(x)$\,, $\forall\,x\in\Xi\,,\,\alpha\in\mathfrak A_{{\sf s}(x)}$\,,
\item[(b)] each $\mathcal T_x$ is an isomorphism $:\mathfrak A_{{\rm s}(x)}\!\to\mathfrak A_{{\rm r}(x)}$\,,
\item[(c)] $\mathcal T_u={\rm id}_{\mathfrak A_u}$\,, $\forall\,u\in U$\,,
\item[(d)] if $(x,y)\in\Xi^{(2)}$ and $(\alpha,y)\in\mathfrak A\rtimes\Xi$\,, then $\big(\mathcal T_y(\alpha),x\big)\in\mathcal A\rtimes\Xi$ and $\mathcal T_{xy}(\alpha)=\mathcal T_x\big[\mathcal T_y(\alpha)\big]$.
\end{enumerate}
\end{defn}

We denote by $C_{{\rm c},\mathscr A}(\Xi)$ the vector space of all the finitely supported functions $F:\Xi\to\bigsqcup_{u\in U}\mathfrak A_u$ satisfying $F(x)\in\mathfrak A_{{\rm r}(x)}$ for every $x\in\Xi$\,. It becomes a $^*$-algebra with the laws
\begin{equation}\label{unaca}
(F\dagger G)(x):=\!\sum_{y\in\Xi^{{\rm r}(x)}}\!F(y)\mathcal T_y\big[G(y^{-1}x)\big]\,,
\end{equation}
\begin{equation}\label{binaca}
F^{\dagger}(x):=\mathcal T_{x}\big[F(x^{-1})\big]^*.
\end{equation}
Completing in the Hahn-type norm
\begin{equation}\label{trinaca}
\p\!F\!\p_{\ell^{\infty,1}_\mathscr A(\Xi)}:=\max\Big\{\sup_{u\in U}\sum_{{\rm r}(x)=u}\!\p\!F(x)\!\p_{\mathfrak A_u}\,,\,\sup_{u\in U}\sum_{{\rm s}(x)=u}\!\p\!F(x^{-1})\!\p_{\mathfrak A_u}\!\Big\}\,,
\end{equation}
one finds a Banach $^*$-algebra $\ell^{\infty,1}_\mathscr A(\Xi)$\,. Its enveloping $C^*$-algebra is known as {\it the crossed product associated to the action $(\mathscr A,p,\mathcal T,\Xi)$}\,, cf \cite{MW1} for example. 

\begin{cor}\label{cuaktiuni}
Let $(\mathscr A,p,\mathcal T,\Xi)$ be an action of the discrete rigidly symmetric groupoid $\Xi$ on the $C^*$-bundle $\mathscr A$. The Banach $^*$-algebra $\ell^{\infty,1}_\mathscr A(\Xi)$ is symmetric.
\end{cor}

\begin{proof}
The proof consists in putting the objects in the framework of Fell bundles over groupoids and apply Theorem \ref{principala}. We follow closely \cite{MW}.

\smallskip
We endow the bundle $\mathscr C:=\mathscr A\rtimes\Xi$ with the operations
$$
(\alpha,x)\bu(\beta,y):=\big(\alpha\mathcal T_x(\beta),xy\big)\,,
$$
$$
(\alpha,x)^\bu:=\big(\mathcal T_{x^{-1}}(\alpha^*),x^{-1}\big)
$$
and get a Fell bundle over $\Xi$\,. A section is now a map $\Phi:\Xi\to\mathscr A\rtimes\Xi$ such that 
$$
\Phi(x)\equiv\big(\varphi(x),x\big)\in\mathfrak C_x=\mathfrak A_{{\rm s}(x)}\!\times\!\{x\}\,,\quad\forall\,x\in\Xi\,.
$$
With such notations, on $\ell^{\infty,1}(\Xi\!\mid\!\mathscr C)$ the relevant mathematical structure is
$$
\begin{aligned}
(\Phi\star\Psi)(x)&=\sum_{yz=x}\big(\varphi(y),y\big)\bu\big(\psi(z),z\big)\\
&=\sum_{yz=x}\big(\varphi(y)\mathcal T_y[\psi(z)],yz\big)\\
&=\sum_{y\in\Xi^{{\rm r}(x)}}\big(\varphi(y)\mathcal T_y[\psi(y^{-1}x)],x\big)\,,
\end{aligned}
$$
$$
\Phi^\star(x)=\big(\varphi(x^{-1}),x^{-1}\big)^\bu=\big(\mathcal T_x\big[\varphi(x^{-1})^\ast\big],x\big)
$$
and
$$
\begin{aligned}
\p\!\Phi\!\p_{\ell^{\infty,1}(\Xi\mid\mathscr C)}=\max\Big\{&\sup_{u\in U}\sum_{{\rm r}(x)=u}\!\p\!(\varphi(x),x)\!\p_{\mathfrak A_{{\rm r}(x)}\times\{x\}}\,,\,\\
&\sup_{u\in U}\sum_{{\rm s}(x)=u}\!\p\!(\varphi(x^{-1}),x^{-1})\!\p_{\mathfrak A_{{\rm s}(x)}\times\{x\}}\!\Big\}\,.
\end{aligned}
$$
Comparing this with \eqref{unaca},\,\eqref{binaca} and \eqref{trinaca}, respectively, it becomes obvious that the Banach $^*$-algebras $\ell^{\infty,1}_\mathscr A(\Xi)$ and $\ell^{\infty,1}(\Xi\!\mid\!\mathscr C)$ are isomorphic. Applying Theorem \ref{principala} finishes the proof.
\end{proof}

\begin{ex}\label{stir}
By {\it a left groupoid action} of our discrete groupoid $\Xi$ on the Hausdorff locally compact space $\Si$ we understand  a pair $(\rho,\tau)$\,, where $\rho:\Si\to U=\Xi^{(0)}$ is continuous surjection, defining the topological subspace 
$$
\Si*\Xi\equiv\Si\rtimes_\tau^\rho\Xi:=\{(\si,x)\in\Si\!\times\!\Xi\!\mid\!\rho(\si)={\rm s}(x)\}\subset \Si\!\times\!\Xi\,,
$$
and a continuous map $\Si\rtimes_\tau^\rho\Xi\ni(\si,x)\to \tau(\si,x)\equiv\tau_x(\si)\in\Si$ satifying the axioms
\begin{enumerate}
\item[(a)] $\tau_{\rho(\si)}(\si)=\si\,,\ \forall\,\si\in\Si\,,$
\item[(b)] if $(x,y)\in\Xi^{(2)}$ and $(\si,y)\in\Si\rtimes_\tau^\rho\Xi$\,,\\ 
then $\big(\tau_y(\si),x\big)\in\Si\rtimes_\tau^\rho\Xi$ and $\tau_{xy}(\si)=\tau_x\big[\tau_y(\si)\big]$\,.
\end{enumerate}
To a groupoid action $(\Xi,\rho,\tau,\Si)$ one associates (cf.\;\cite{MW1}) {\it the transformation} (or {\it action} or {\it crossed product}) {\it groupoid}, with unit space which can be identified to $\Si$,  by endowing $\Si\rtimes_\tau^\rho\Xi$ with the structure maps
$$
\big(\tau_x(\si),y\big)(\si,x):=(\si,yx)\,,\quad(\si,x)^{-1}:=\big(\tau_x(\si),x^{-1}\big)\,,
$$
$$
{\mathfrak s}(\si,x):=(\si,\rho(\si))\equiv\si\,,\quad{\mathfrak r}(\si,x):=\big(\tau_x(\si),{\rm r}(x)\big)\equiv\tau_x(\si)\,.
$$
For each $u\in U$ one defines the open set $\Si_u:=\rho^{-1}(u)\subset\Si$ and then the Abelian $C^*$-algebra $\mathfrak A_u\!:=C_0(\Si_u)$\,. The homeomorphisms $\tau_{x^{-1}}:\Si_{{\rm r}(x)}\to\Si_{{\rm s}(x)}$ induce by composition isomorphisms $\mathcal T_x:\mathfrak A_{{\rm s}(x)}\to\mathfrak A_{{\rm r}(x)}$\,. One gets an action $\mathcal T$ of $\Xi$ on the $C^*$-bundle $\mathscr A=\bigsqcup_{u\in U}\mathfrak A_u$\,. It turns out that $\ell^{\infty,1}_\mathscr A(\Xi)$ has the Hahn algebra $L^{\infty,1}\big(\Si\rtimes^\rho_\tau\Xi\big)$ of the transformation groupoid $\Si\rtimes^\rho_\tau\Xi$ as a quotient. The next result, consequence of this and Corollary \ref{cuaktiuni}, converts the rigid symmetry of a discrete groupoid into the symmetry of any of its possible transformation groupoids (which may be much more complicated). Note that the space $\Si$ needs not be discrete.
\end{ex}

\begin{cor}\label{cuagtiuni}
Let $(\Xi,\rho,\tau,\Si)$ be an action of the discrete rigidly symmetric groupoid $\Xi$ on the locally compact space $\Si$\,. The transformation groupoid $\Si\rtimes_\tau^\rho\Xi$ is symmetric.
\end{cor}

%-------------------------------------------------------------------------------------------------------
\subsection{Behavior of symmetry with respect to  epimorphisms and monomorphisms}\label{minimor}
%-------------------------------------------------------------------------------------------------------

We will work with a {\it morphism of discrete groupoids} $j:\Xi\to\Xi'$, which are assumed at the beginning to have the same unit space $U$ (or that $j|_U:U\to U'$ is a bijection).

\begin{thm}\label{acidborich}
Let $j:\Xi\to\Xi'$ be an surjective morphism of discrete groupoids. If $\Xi$ is symmetric, then $\Xi'$ is symmetric.
\end{thm}

\begin{proof}
We define $\,\mathbb J:C_{\rm c}(\Xi)\to C_{\rm c}(\Xi')$ by
\begin{equation*}\label{haplea}
\big[\mathbb J(\Phi)\big](x'):=\!\sum_{j(x)=x'}\!\Phi(x)\,.
\end{equation*}
It is a linear map satisfying $\,{\rm supp}\big[\mathbb J(\Phi)\big]\subset j\big[{\rm supp}(\Phi)\big]$\,. Since $J$ is onto, $\mathbb J$ is also onto (check that if $\,{\rm supp}(\Phi')=\{x'_0\}$ then $\Phi'$ belongs to the range of $\mathbb J$).

\smallskip
The application $\mathbb J$ is multiplicative:
$$
\begin{aligned}
\big[\mathbb J(\Phi)\star\mathbb J(\Psi)\big](x')&=\sum_{x_1'x_2'=x'}\sum_{j(x_1)=x_1'}\!\!\Phi(x_1)\!\sum_{j(x_2)=x_2'}\!\!\Psi(x_2)\\
&=\sum_{j(x)=x'}\sum_{x_1x_2=x}\!\Phi(x_1)\Psi(x_2)\\
&=\sum_{j(x)=x'}(\Phi\star\Psi)(x)=\big[\mathbb J(\Phi\star\Psi)\big](x')\,.
\end{aligned}
$$

For the involution:
$$
\begin{aligned}
\big[\mathbb J(\Phi)\big]^\star(x')&=\overline{\big[\mathbb J(\Phi)\big]\big(x'^{-1}\big)}=\sum_{j(y)=x'^{-1}}\!\overline{\Phi(y)}\\
&=\sum_{j(x)=x'}\!\overline{\Phi(x^{-1})}=\sum_{j(x)=x'}\Phi^\star(x)=\mathbb J\big(\Phi^\star\big)(x')\,.
\end{aligned}
$$

\smallskip
We check now that $\mathbb J$ is contractive for the Hahn norms. We only treat the ${\rm r}$-part of the Hahn norm; the ${\rm s}$-part is similar:
$$
\begin{aligned}
\p\!\mathbb J(\Phi)\!\p^{\rm r}_{\ell^{\infty,1}(\Xi')}&=\sup_{u'\in U'}\sum_{{\rm r}'(x')=u'}\Big\vert \!\sum_{j(x)=x'}\!\Phi(x)\Big\vert\\
&\le\sup_{u'\in U'}\sum_{{\rm r}'(x')=u'}\!\sum_{\ j(x)=x'}\big\vert\Phi(x)\big\vert\\
&=\sup_{u\in U}\sum_{{\rm r}'(x')=j(u)}\!\sum_{\ j(x)=x'}\big\vert \Phi(x)\big\vert\\
&=\sup_{u\in U}\sum_{{\rm r}(x)=u}\!\vert\Phi(x)\vert\\
&=\,\p\!\Phi\!\p^{\rm r}_{\ell^{\infty,1}(\Xi)}.
\end{aligned}
$$
In the next-to-last step we used the identity ${\rm r}'\circ j=j\circ{\rm r}$ and also the injectivity of $j$\,, to make sure that $j[{\rm r}(x)]=j(u)$ is equivalent to ${\rm r}(x)=u$\,.

\smallskip
Now let $\Phi\in C_{\rm c}(\Xi)$\,, with (finite) support $F\subset\Xi$\,, and define $\Psi\in C_{\rm c}(\Xi)$ to be null if $F\cap j^{-1}[j(x)]=\emptyset$ and
\begin{equation*}\label{lamehuha}
\Psi(x):=\frac{(\mathbb J\Phi)[j(x)]}{\#\big(F\cap j^{-1}[j(x)]\big)}=\frac{\sum_{j(y)=j(x)}\Phi(y)}{\#\big(F\cap j^{-1}[j(x)]\big)}
\end{equation*}
if $F\cap j^{-1}[j(x)]\ne\emptyset$\,. A short computation shows that $\mathbb J(\Psi)=\mathbb J(\Phi)$\,. Additionally
\begin{equation*}\label{lamezuza}
\begin{aligned}
\p\!\Psi\!\p^{\rm r}_{\ell^{\infty,1}(\Xi)} &=\sup_{u\in U}\sum_{{\rm r}(x)=u}\frac{|(\mathbb J\Phi)[j(x)]|}{\#\big(F\cap j^{-1}[j(x)]\big)}\\
&=\sup_{u\in U}\sum_{{\rm r'}[j(x)]=u}\frac{|(\mathbb J\Phi)[j(x)]|}{\#\big(F\cap j^{-1}[j(x)]\big)}\\
&=\sup_{u\in U}\sum_{{\rm r'}(x')=u}|(\mathbb J\Phi)(x')|=\,\p\!\mathbb J(\Psi)\!\p_{\ell^{\infty,1}(\Xi')}
\end{aligned}
\end{equation*}
so
\begin{equation}\label{mameluza}
\p\!\Psi\!\p_{\ell^{\infty,1}(\Xi)}\,=\,\p\!\mathbb J(\Psi)\!\p_{\ell^{\infty,1}(\Xi')}\,=\,\p\!\mathbb J(\Phi)\!\p_{\ell^{\infty,1}(\Xi')}.
\end{equation}
By \eqref{mameluza} and the contractivity of $\mathbb J$\,, one concludes that
\begin{equation*}\label{labeluza}
\p\!\mathbb J(\Phi)\!\p_{\ell^{\infty,1}(\Xi')}=\min\big\{\!\p\!\Psi\!\p_{\ell^{\infty,1}(\Xi)}\,\big\vert\, \mathbb J(\Psi)=\mathbb J(\Phi)\big\}
\end{equation*}
(the quotient norm). It follows that $\mathbb J$ extends to a surjective contraction from $\ell^{\infty,1}(\Xi)$ to $\ell^{\infty,1}(\Xi')$\,, which is a $^*$-morphism as shown in the first part of the proof. Or, to put it differently, $\ell^{\infty,1}(\Xi')$ identifies as a Banach $^*$-algebra with the quotient $\ell^{\infty,1}(\Xi)/\ker\mathbb J$\,.

\smallskip
The statement about symmetry follows, since by \cite[Th.\,10.4.4]{Pa2} the quotient of a symmetric Banach $^*$-algebra through a closed bi-sided self-adjoint ideal is symmetric.
\end{proof}

We finish this section with a result concerning monomorphisms. In fact, in the injective case, it is possible to work with {\it morphism of Fell bundles} 
$$
(J,j):(\mathscr C,q,\Xi)\to\big(\mathscr C',q',\Xi'\big)
$$
over discrete groupoids. Thus $j:\Xi\to\Xi'$ is a groupoid morphism and $J:\mathscr C\to\mathscr C'$ is a continuous map such that $q'\circ J=j\circ q$\,. It is assumed that the induced applications 
$$
J_x:\mathfrak C_x\to\mathfrak C'_{j(x)}\,,\quad x\in\Xi
$$ 
are linear, one has 
$$
J_{xy}(a\bu b)=J_x(a)\bu J(b)\,,\quad\forall\,a\in\mathfrak C_x\,,b\in\mathfrak C_y\,,(x,y)\in\Xi^{(2)}
$$
and $J_{x^{-1}}\big(a^\bu\big)=J_x(a)^\bu$ if $a\in\mathfrak C_x$\,.

\smallskip
It follows that $J_u:\mathfrak C_u\to\mathfrak C'_{j(u)}$ are $C^*$-morphisms for all the units $u\in U$\,, thus they are contractions. Then for every $b\in\mathfrak C_x$\,, $x\in\Xi$\,, one gets
$$
\begin{aligned}
\p\!J_x(b)\!\p_{\mathfrak C'_{j(x)}}^2&=\,\p\!J_x(b)^\bu\bu J_x(b)\!\p_{\mathfrak C'_{{\rm s}'[j(x)]}}\\
&=\,\p\!J_{{\rm s}(x)}(b^\bu\bu b)\!\p_{\mathfrak C'_{j[{\rm s}(x)]}}\\
&\le\,\p\!b^\bu\bu b\!\p_{\mathfrak C_{{\rm s}(x)}}\\
&=\,\p\!b\!\p_{\mathfrak C_x}^2\,,
\end{aligned}
$$
therefore all the operators $J_x:\mathfrak C_x\to\mathfrak C'_{j(x)}$ are contractive. In addition, if $J$ is injective (this is equivalent to its injectivity on the algebras corresponding to units), they are all isometries.

\begin{thm}\label{acidboroch}
Let $(J,j):(\mathscr C,q,\Xi)\to\big(\mathscr C',q',\Xi'\big)$ be a monomorphism of Fell bundles.
 If $\ell^{\infty,1}\big(\Xi'\!\mid\!\mathscr C'\big)$ is symmetric, then $\ell^{\infty,1}(\Xi\!\mid\!\mathscr C)$ is symmetric.
\end{thm}

\begin{proof}
Let $\,\mathbf J:\ell^{\infty,1}(\Xi\!\mid\!\mathscr C)\to\ell^{\infty,1}\big(\Xi'\!\mid\!\mathscr C'\big)$\,, where $\big[\mathbf J(\Phi)\big](x'):=0_{\mathfrak C'_{x'}}$ if $x'$ does no belong to the range of $j$ and
$$
\big[\mathbf J(\Phi)\big](x'):=J_x\big[\Phi(x)\big]\quad{\rm if}\ \ \ j(x)=x'.
$$
It is straightforward to check that $\mathbf J$ is an isometric $^*$-morphism, so $\ell^{\infty,1}\big(\Xi\!\mid\!\mathscr C\,\big)$ embeds as a closed $^*$-subalgebra of $\ell^{\infty,1}\big(\Xi'\!\mid\!\mathscr C'\big)$\,. For the isometry, we need the remark made above, saying that $J$ is isometric if it is injective. The conclusion follows recalling that closed $^*$-subalgebras of symmetric Banach $^*$-algebras are themselves symmetric.
\end{proof}

This can be applied, for instance, if $\Xi=\Xi'$ and $\mathscr C$ is a Fell sub-bundle of $\mathscr C'$.

\begin{cor}\label{acidboricoci}
Suppose that $\Xi$ is a subgroupoid of the discrete groupoid $\Xi'$.
\begin{enumerate}
\item[(i)] If $\Xi'$ is symmetric, then $\Xi$ is symmetric.
\item[(ii)] If $\Xi'$ is hypersymmetric, then $\Xi$ is hypersymmetric.
\end{enumerate}
\end{cor}

\begin{proof}
(i) One gets the result about symmetry by taking the trivial Fell bundles $\mathscr C=\mathscr C':=\mathbb C\!\times\!\Xi$\,.

\smallskip
(ii) For hypersymmetry one starts with an arbitrary Fell bundle $\mathscr C$ over $\Xi$\,, extends it somehow over the entire $\Xi'$ (by zero fibers, for instance) and then applies Theorem \ref{acidboroch}.
\end{proof}

Corollary \ref{acidboricoci} can be applied to various interesting subgroupoids, as the isotropy groupoid for instance, or as the (maybe non-invariant) restriction 
$$
\Xi^M_M:=\{x\in\Xi\!\mid\!{\rm s}(x)\in M\,,{\rm r}(x)\in M\}
$$ 
with $M\subset U$\,. For $M=\{u\}$ one gets the isotropy group $\Xi_u^u$\,. We point out that in \cite[Prop.\,4.2]{AO} it is shown that the isotropy subgroups of a symmetric groupoid are symmetric, under some assumptions on the big groupoid (ample, with compact unit space and satisfying an extra technical condition). A twist is also considered.

%-------------------------------------------------------------------------------------------------------
\section{The case of locally compact groups}\label{flaconir}
%-------------------------------------------------------------------------------------------------------

%-------------------------------------------------------------------------------------------------------
\subsection{The main results for locally compact groups}\label{flocinor}
%-------------------------------------------------------------------------------------------------------

All over the remaining part of the article, $\G$ will be a (Hausdorff) locally compact group with unit $\e$ and left Haar measure $d\mu(x)\equiv dx$\,. 
We are also given a Fell bundle $\mathscr C\!=\bigsqcup_{x\in\G}\mathfrak C_x$ over $\G$\,.
Its sectional $L^1(\G\,\vert\,\mathscr C)$ Banach $^*$-algebra is the completion of the space $C_{\sf c}(\G\,\vert\,\mathscr C)$ of continuous sections with compact support. Its (universal) $C^*$-algebra its denoted by ${\rm C^*}(\G\,\vert\,\mathscr C)$\,. For the general theory of Fell bundles over groups we followed \cite[Ch.\,VIII]{FD2}, to which we refer for details. The case of a discrete group and associated graded $C^*$-algebras is developed in \cite{Ex3}. We only recall the product on $L^1(\G\,\vert\,\mathscr C)$
\begin{equation}\label{broduct}
\big(\Phi*\Psi\big)(x)=\int_\G \Phi(y)\bu \Psi(y^{-1}x)\,\d y
\end{equation}
and its involution
\begin{equation}\label{inwol}
\Phi^*(x)=\Delta(x^{-1})\,\Phi(x^{-1})^\bu\,,
\end{equation}
in terms of the operations $\big(\bu,^\bu\big)$ on the Fell bundle.

\smallskip
We are going to make use both of $\mathscr C$ and its discrete version; the following result is obvious.

\begin{lem}\label{memix}
Let $\mathscr C\!=\bigsqcup_{x\in\G}\mathfrak C_x$ be a Fell bundle over the locally compact group $\G\,$ and $\G^{\rm dis}$ the same group with the discrete topology\,. We denote by $\mathscr C^{\rm dis}$ the space $\mathscr C$ with the disjoint union topology defined by the fibres $\big\{\mathfrak C_x\,\big\vert\,x\in\G\big\}$ (a subset is open if and only if its intersection with every $\mathfrak C_x$ is open). Then $\mathscr C^{\rm dis}$ is a Fell bundle over $\G^{\rm dis}$, called {\rm the discretization of the initial one}. One has
\begin{equation*}\label{continix}
C_{\rm c}\big(\G^{\rm dis}\,\vert\,\mathscr C^{\rm dis}\big)=\big\{\Phi:\G\to\mathscr C\,\big\vert\,\Phi(x)\in\mathfrak C_x\,,\,\forall\,x\in\G\ \,{\rm and\,\ supp}(\Phi)\ {\rm is\ finite}\big\}\,,
\end{equation*}
\begin{equation*}\label{sumix}
\begin{aligned}
L^1\big(\G^{\rm dis}\,\vert\,\mathscr C^{\rm dis}\big)&\equiv\ell^1(\G\,\vert\,\mathscr C)\\
&=\big\{\Phi:\G\to\mathscr C\,\big\vert\,\Phi(x)\in\mathfrak C_x\,,\,\forall\,x\in\G\ \,{\rm and}\ \sum_{x\in\G}\!\p\!\Phi(x)\!\p\,<\infty\big\}\,,
\end{aligned}
\end{equation*}
\begin{equation*}\label{diferix}
C^*\big(\G^{\rm dis}\,\vert\,\mathscr C^{\rm dis}\big)= {\rm the\ enveloping}\ C^*{\rm -\,algebra\ of}\ \ell^1(\G\,\vert\,\mathscr C)\,.
\end{equation*}
\end{lem}

The critical technical result is the next theorem. It will be proven in Section \ref{proof}, on the lines of \cite{Po,FL}, where only particular cases of Fell bundles have been treated. Some technical complications which have to be overcome are due to the fact that our $L^1$-sections do not take values in a single Banach space.

\begin{thm}\label{cruciala}
Let $\mathscr C$ be a Fell bundle over the locally compact group $\G\,$.  
If $\ell^1(\G\,\vert\,\mathscr C)$ is symmetric, then $L^1(\G\,\vert\,\mathscr C)$ is also symmetric.
\end{thm}

We state now two main results of the paper. They follow immediately from the two theorems above, where ${\sf H}=\G^{\rm dis}$ and $\mathscr D$ is the discretization of $\mathscr C$.

\begin{thm}\label{theoremix}
Let $\mathscr C$ be a Fell bundle  over the locally compact group $\G$\, for which the discrete group $\G^{\rm dis}$ is rigidly symmetric. Then $L^1(\G\,\vert\,\mathscr C)$ is a symmetric Banach $^*$-algebra.
\end{thm}

If one only knows that $\G^{\rm dis}$ is symmetric, there is still a result if we ask more on the $C^*$-algebraic side.

\begin{thm}\label{theoremix1}
Let $\mathscr C$ be a Fell bundle  over the locally compact group $\G$ for which the discrete group $\G^{\rm dis}$ is symmetric. Suppose that $C^*\big(\G^{\rm dis}\vert\mathscr C^{\rm dis}\big)$ is a type I $C^*$-algebra. Then $L^1(\G\,\vert\,\mathscr C)$ is symmetric. 
\end{thm}

%??????????????????????????????????????????
\subsection{Discretization of representations}\label{pispis}
%???????????????????????????????????????????

Our proof of Theorem \ref{cruciala}, to be found in the next section, relies on a discretization procedure in the setting of Fell bundles, that we now present. 

\smallskip
Let $\Gamma:L^1(\G\,\vert\,\mathscr C)\to\mathbb B(\mathcal E)$ be a non-trivial algebraically irreducible representation on the vector space $\mathcal E$. Even for general Banach algebras, it is known how to turn it into a Banach space representation. In our case, for $\xi_0\in\mathcal E$, the norm 
$$
\norm{\xi}_{\mathcal E}:=\inf\big\{\norm{\Phi}_{L^1(\G\,\vert\,\mathscr C)}\!\mid \Gamma(\Phi)\xi_0=\xi\big\}
$$ 
makes $\mathcal E$ a Banach space and $\Gamma$ a contractive representation. If $\mathcal E$ was already a Banach space, this new norm is equivalent to the previous one, so we lose nothing by assuming that this is the norm of $\mathcal E$. 

\smallskip
We will define a representation $\Gamma^{\rm dis}$ of $\ell^1\big(\G\,\vert\,\mathscr C\big)\equiv L^1\big(\G^{\rm dis}\,\vert\,\mathscr C^{\rm dis}\big)$ on $\mathcal E$, called its {\it discretization}. The procedure consists in disintegrating the initial representation, reinterpreting it in the discrete setting and integrating it back. The new one will act in the same space, but it will be very different. 

\smallskip
By algebraic irreducibility, we may write every $\xi\in\mathcal{E}$ as $\xi=\Gamma(\Psi)\xi_0$\,. Using this form, the representation $\Gamma$ induces bounded operators $\{\gamma(a)\}_{a\in \mathscr C}$\,, given by 
\begin{equation*}\label{srtass}
\gamma(a)\Big[\Gamma(\Psi)\xi_0\Big]:=\Gamma(a\cdot \Psi)\xi_0
\end{equation*}
where, in terms of the bundle projection $q:\mathscr C\to\G$\,, we set
\begin{equation*}\label{srtass1}
(a\cdot \Psi)(x):=a\bu\Psi\big(q(a)^{-1}x\big)\in\mathfrak C_x\,,\quad\forall\,x\in\G\,.
\end{equation*}
We are thinking of $a\in \mathscr C$ as a multiplier of $L^1(\G\,\vert\,\mathscr C)$\,, that we treat directly, without indicating references. So we provide a direct proof of the following result:

\begin{lem}\label{repr}
The map $\gamma:\mathscr C\to \mathbb B(\mathcal E)$ is a well-defined, contractive, Banach representation of the Fell Bundle $\mathscr C$ and does not depend on the choices.
\end{lem}

\begin{proof}
Let $\xi=\Gamma(\Psi)\xi_0=\Gamma(\Phi)\xi_1$ and let $\big\{\Upsilon_i\,\big\vert\,i\in I\big\}\subset L^1(\G\,\vert\,\mathscr C)$ be some approximate unit. Then one has 
\begin{align*}
0&= \Gamma(a\cdot \Upsilon_i)(\xi-\xi) \\
&=\Gamma(a\cdot \Upsilon_i)\big(\Gamma(\Psi)\xi_0-\Gamma(\Phi)\xi_1\big) \\
&=\Gamma(a\cdot\big[\Upsilon_i*\Psi\big])\xi_0-\Gamma(a\cdot\big[\Upsilon_i*\Phi\big])\xi_1
\end{align*} So by passing to the limit we get the equality 
$$
\gamma(a)\xi=\Gamma(a\cdot\Psi)\xi_0=\Gamma(a\cdot\Phi)\xi_1.
$$
Let us proceed to check that $\gamma$ has the defining properties of a representation, starting by $\gamma(a)\gamma(b)=\gamma(ab)$ for every $a,b\in\mathscr C$\,: 
\begin{align*}
\gamma(a)\gamma(b)\big[\Gamma(\Psi)\xi_0\big]&=\Gamma\big(ab\bu\Psi(q(b)^{-1}q(a)^{-1}\,\cdot\,)\big)\xi_0 \\ &=\Gamma\big(ab\bu\Psi(q(ab)^{-1}\,\cdot\,)\big)\xi_0 \\
&=\gamma(ab)\big[\Gamma(\Psi)\xi_0\big].
\end{align*}
We now show that $\norm{\gamma(a)}_{\mathbb{B}(\mathcal{E})}\leq\norm{a}_{\mathfrak C_{q(a)}}$\,. By the definition of $\norm{\cdot}_{\mathcal E}$\,, there exists $\{\Phi_n\}_{n\in\mathbb N}\subset L^1(\G\,\vert\,\mathscr C)$ for which $\Gamma(\Phi_n)\xi_0=\xi$ and $\norm{\Phi_n}_{L^1(\G\,\vert\,\mathscr C)}\to\norm{\xi}_{\mathcal E}$\,. Then 
$$
\gamma(a)\xi=\gamma(a)\Gamma(\Phi_n)\xi_0=\Gamma\big(a\cdot \Phi_n\big)\xi_0\,,
$$ 
which yields $\norm{\gamma(a)\xi}_{\mathcal E}\leq \norm{a\cdot \Phi_n}_{L^1(\G\,\vert\,\mathscr C)}$. But 
\begin{align*}
\norm{a\cdot \Phi_n}_{L^1(\G\,\vert\,\mathscr C)}&=\int_\G\,\norm{a\bu\Phi_n(q(a)^{-1}x)}_{\mathfrak C_x}\d x \\
&\leq \int_\G\,\norm{a}_{\mathfrak C_{q(a)}}\norm{\Phi_n(q(a)^{-1}x)}_{\mathfrak C_{q(a)^{-1}x}}\d x \\
&= \norm{a}_{\mathfrak C_{q(a)}}\int_\G\,\norm{\Phi_n(x)}_{\mathfrak C_{x}}\d x \\
&= \norm{a}_{\mathfrak C_{q(a)}}\norm{\Phi_n}_{L^1(\G\,\vert\,\mathscr C)}.
\end{align*} 
So 
$$
\norm{a\cdot \Phi_n}_{L^1(\G\,\vert\,\mathscr C)}\leq \norm{a}_{\mathfrak C_{q(a)}}\norm{\Phi_n}_{L^1(\G\,\vert\,\mathscr C)} \to \norm{a}_{\mathfrak C_{q(a)}}\norm{\xi}_{\mathcal E}\,,
$$ 
hence $\norm{\gamma(a)\xi}_{\mathcal E}\leq \norm{a}_{\mathfrak C_{q(a)}}\norm{\xi}_{\mathcal E}$ and $\norm{\gamma(a)}_{\mathbb B(\mathcal E)}\leq \norm{a}_{\mathfrak C_{q(a)}}$. 

\smallskip
For the continuity of $a\mapsto \gamma(a)\xi$, write (again) $\xi=\Gamma(\Psi)\xi_0$\, and observe that
\begin{align*}
\gamma(b)\xi-\gamma(a)\xi&=\gamma(b)\Gamma(\Psi)\xi_0-\gamma(a)\Gamma(\Psi)\xi_0\\ 
&=\Gamma\big(b\bu\Psi(q(b)^{-1}\,\cdot\,)-a\bu\Psi(q(a)^{-1}\,\cdot\,)\big)\xi_0\\
&=\Gamma\big(b\cdot \Psi-a\cdot \Psi\big)\xi_0\,.
\end{align*} 
In consequence, the continuity of $a\mapsto \gamma(a)\xi$ follows from the continuity of $a\mapsto a\cdot \Psi\in L^1(\G\,\vert\,\mathscr C)$\,. 
\end{proof}

\begin{rem}\label{details}
The initial representation $\Gamma$ can be recovered from the operators $\{\gamma(a)\}_{a\in \mathscr C}$ by 
\begin{equation}\label{froc}
\Gamma(\Phi)\xi=\int_\G \gamma\big(\Phi(x)\big)\xi\,\d x\,.
\end{equation}
Indeed, writing $\xi\in\mathcal{E}$ as $\xi=\Gamma(\Psi)\xi_0$ we get 
\begin{align*}
\int_\G \gamma\big(\Phi(x))\xi\,\d x&=\int_\G \gamma\big(\Phi(x)\big)\big[\Gamma(\Psi)\xi_0\big]\,\d x \\
&=\int_\G \Gamma\big(\Phi(x)\bu \Psi(x^{-1}\,\cdot\,)\big)\xi_0\,\d x \\
&=\Gamma\Big(\int_\G \Phi(x)\bu \Psi\big(x^{-1}\cdot\big)\,\d x\Big)\xi_0 \\
&=\Gamma(\Phi*\Psi)\xi_0=\Gamma(\Phi)\big[\Gamma(\Psi)\xi_0\big]=\Gamma(\Phi)\xi\,.
\end{align*} 
This is part of the so called 'integration/desintegration' theory for Fell Bundles.
\end{rem} 

\begin{defn}\label{discretization}
With the assumptions given above, the {\it discretization} of the representation $\Gamma:L^1(\G\,\vert\,\mathscr C)\to\mathbb B(\mathcal E)$ is $\Gamma^{\rm dis}\!:\ell^1\big(\G\,\vert\,\mathscr C\big)\to \mathbb B(\mathcal E)$ given by 
\begin{equation}\label{frac}
\Gamma^{\rm dis}(\varphi)\xi:=\sum_{x\in\G}\gamma\big(\varphi(x)\big)\xi\,,\quad\textup{ for}\ \varphi\in \ell^1\big(\G\,\vert\,\mathscr C\big)\,.
\end{equation}
\end{defn}

\begin{rem}\label{same}
$\Gamma^{\rm dis}$ is an irreducible representation of a cross-sectional algebra, so it also induces operators 
$\{\gamma^{\rm dis}(a)\}_{a\in \mathscr C}$. We point out that $\gamma^{\rm dis}(a)=\gamma(a)$\,, which follows from the existing definitions by a straightforward computation.
%To check this, write $\xi=\Gamma^{\rm dis}(\psi)\xi_0$ and compute
%\begin{align*}
%\gamma^{\rm dis}(a)\xi&=\gamma^{\rm dis}(a)\Gamma^{\rm dis}(\psi)\xi_0 \\
%&=\Gamma^{\rm dis}\big(a\psi(q(a)^{-1}\,\cdot\,)\big)\xi_0 \\
%&=\sum_{y\in\G} \gamma(a\psi(q(a)^{-1}y))\xi_0 \\ 
%&=\sum_{y\in\G} \gamma(a)\gamma(\psi(q(a)^{-1}y))\xi_0 \\ 
%&=\gamma(a)\sum_{z\in\G} \gamma(\psi(z))\xi_0 \\
%&=\gamma(a)\Gamma^{\rm dis}(\psi)\xi_0 \\
%&=\gamma(a)\xi.
%\end{align*}
On the other hand, the discrete setting allows a simpler treatment (multipliers are no longer needed). The Fell bundle $\mathscr C^{\rm dis}$ injects isometrically in $\ell^1\big(\G\,\vert\,\mathscr{C}\big)$ by
\begin{equation}\label{lagarta}
\big[\mu(a)\big](x):=\delta_{q(a),x}\,a\,,\quad\forall\,a\in\mathscr C,x\in\G
\end{equation}
(which may be written simply $\mu(a)=a\delta_{q(a)}$)\,, and then we set $\gamma^{\rm dis}:=\Gamma^{\rm dis}\!\circ\mu$\,. To check that this is the same object, one proves immediately that $\mu(a)*\psi=a\cdot\psi$ and then apply $\Gamma^{\rm dis}$\,:
$$
\Gamma^{\rm dis}\big[\mu(a)\big]\Gamma^{\rm dis}(\psi)\xi_0=\Gamma^{\rm dis}\big[\mu(a)*\psi\big]\xi_0=\Gamma^{\rm dis}(a\cdot\psi\big)\xi_0=\gamma(a)\Gamma^{\rm dis}(\psi)\xi_0\,.
$$
\end{rem}

To obtain the main result of this section, which is Proposition \ref{A3}, let us state a couple of lemmas: 

\begin{lem}\label{A2}
Let $T:\mathcal X\to \mathcal Y$ be a bounded operator between Banach spaces. If $\,T$ satisfies the condition 
$$
\forall\,y\in\mathcal Y,\,\forall\,\epsilon>0\,,\,\exists\,x\in\mathcal X\textup{ such that } \max\{\norm{x}-\norm{y},\norm{Tx-y}\}\le\epsilon\,,
$$ \,
then it must be surjective.
\end{lem}

\begin{proof}
This is Lemma A2 from \cite{BB1}; see also \cite[pag. 193]{Po}.
\end{proof}

\begin{lem}\label{ubiforti}
For every $\eta\in\mathcal E$ and $\epsilon>0$ there exists $\varphi\in\ell^1(\G\,\vert\,\mathscr C)$ such that 
\begin{equation}\label{virst}
\p\!\varphi\!\p_{\ell^1(\G\,\vert\,\mathscr C)}\,\le\,\p\!\eta\!\p_{\mathcal E}+\epsilon\,,\quad\p\!\Gamma^{\rm dis}(\varphi)\xi_0-\eta\!\p_{\mathcal E}\,\le\epsilon\,.
\end{equation}
\end{lem}

\begin{proof}
Let $0<\delta$\,, to be fixed later. Due to the definition of the norm in $\mathcal E$ and the density of $C_{\rm c}(\G\,\vert\,\mathscr C)$ in $L^1(\G\,\vert\,\mathscr C)$\,, there is an element $\Phi\in C_{\rm c}(\G\,\vert\,\mathscr C)$ such that
\begin{equation}\label{zecund}
\p\!\Phi\!\p_{L^1(\G\,\vert\,\mathscr C)}\,\le\,\p\!\eta\!\p_{\mathcal E}+\delta\,,\quad\p\!\Gamma(\Phi)\xi_0-\eta\!\p_{\mathcal E}\,\le\delta\p\!\xi_0\!\p_{\mathcal E}.
\end{equation}
Recall that, for such a section, the maps $x\mapsto\,\p\!\Phi(x)\!\p_{\mathfrak C_x}$ and $x\mapsto \gamma\big(\Phi(x)\big)\xi_0$ are uniformly continuous. The support of $\Phi$ will be denoted by $\Si$\,. Let us fix an open relatively compact neighborhood $V$ of $\e\in\G$ such that 
\begin{equation}\label{ciorsica}
\p\!\gamma\big(\Phi(xu)\big)\xi-\gamma\big(\Phi(x)\big)\xi\!\p_{\mathcal E}\,\le\delta/\mu(\Si)\,,\quad\forall\,x\in\G\,, u\in V
\end{equation}
and
\begin{equation}\label{kiorsica}
\big\vert\!\p\!\Phi(yu)\!\p_{\mathfrak C_{yu}}-\p\!\Phi(y)\!\p_{\mathfrak C_y}\!\big\vert\le\delta/\mu(\Si)\,,\quad\forall\,x\in\G\,, u\in V.
\end{equation}
Choose $x_1,\dots,x_m\in\G$ such that $\Si\subset\bigcup_{j=1}^m x_jV.$ Setting $M_1:=x_1V\cap\Si$\,, and then inductively 
$$
M_k:=\big(x_k V\cap\Si\big)\!\setminus\!\bigcup_{j<k}M_j\,,\quad k=2,\dots,m\,,
$$
one gets a measurable partition of $\Si$\,. The solution to our problem is expected to be the finitely supported function
\begin{equation*}\label{disperat}
\varphi:=\sum_{j=1}^m\mu(M_j)\Phi(x_j)\delta_{x_j}\,.
\end{equation*}
Its single non-null values are $\varphi(x_j)=\mu(M_j)\Phi(x_j)\in\mathfrak C_{x_j}$\,, $j=1,\dots,m$\,, so $\varphi\in\ell^1(\G\!\mid\!\mathscr C)$\,. One has
$$
\begin{aligned}
\p\!\Phi\!\p_{L^1(\G\,\vert\,\mathscr C)}&=\sum_{j=1}^m\int_{M_j}\!\!\p\!\Phi(x)\!\p_{\mathfrak C_x}\!d\mu(x)\\
&\ge\sum_{j=1}^m\big[\p\!\Phi(x_j)\!\p_{\mathfrak C_{x_j}}\!\!-\delta/\mu(\Si)\,\big]\int_{M_j}\!d\mu(x)\\
&=\sum_{j=1}^m\p\!\Phi(x_j)\!\p_{\mathfrak C_{x_j}}\mu(M_j)-\delta\\
&=\,\p\!\varphi\!\p_{\ell^1(\G\,\vert\,\mathscr C)}\!-\,\delta\,.
\end{aligned}
$$
For the inequality we used \eqref{kiorsica} with $y=x_k$ and the fact that $M_j\subset x_jV$. Combining this with the first identity in \eqref{zecund}, we get
\begin{equation*}\label{primpas}
\p\!\varphi\!\p_{\ell^1(\G\,\vert\,\mathscr C)}\,\le\,\p\!\eta\!\p_{\mathcal E}+2\delta\,.
\end{equation*} 
To reach the second condition in \eqref{virst}, by \eqref{zecund}, we only need to control the difference $\Gamma(\Phi)\xi_0-\Gamma^{\rm dis}(\varphi)\xi_0$\,, using the formulae \eqref{froc} and \eqref{frac}, which in our case become
\begin{equation*}\label{fric}
\Gamma(\Phi)\xi_0=\sum_{j=1}^m\int_{M_j}\!\gamma\big(\Phi(x)\big)\xi_0\,\d x
\end{equation*}
and 
\begin{equation*}\label{fruc}
\Gamma^{\rm dis}(\varphi)\xi_0=\sum_{j=1}^m \mu(M_j)\gamma\big(\Phi(x_j)\big)\xi_0=\sum_{j=1}^m \int_{M_j}\!\gamma\big(\Phi(x_j)\big)\xi_0\,\d x.
\end{equation*}
Therefore, by \eqref{ciorsica}
\begin{align*}
\norm{\Gamma(\Phi)\xi_0-\Gamma^{\rm dis}(\varphi)\xi_0}_{\mathcal E}&=\norm{\sum_{j=1}^m \int_{M_j}\!\gamma\big(\Phi(x)\big)\xi_0-\gamma\big(\Phi(x_j)\big)\xi_0\,\d x}_{\mathcal E} \\
&\leq \sum_{j=1}^m \int_{M_j}\norm{ \gamma\big(\Phi(x)\big)\xi_0-\gamma\big(\Phi(x_j)\big)\xi_0}_{\mathcal E}\,\d x \\
&\leq \sum_{j=1}^m \delta\mu(M_j)/\mu(\Sigma) \\
&=\delta\,.
\end{align*} So we have \begin{align*}
\p\!\Gamma^{\rm dis}(\varphi)\xi_0-\eta\!\p_{\mathcal E}\,&\le\, \p\!\Gamma(\Phi)\xi_0-\Gamma^{\rm dis}(\varphi)\xi_0\!\p_{\mathcal E}+\p\!\Gamma(\Phi)\xi_0-\eta\!\p_{\mathcal E} \\
&\le\delta(1+\norm{\xi_0}_{\mathcal E})\,.
\end{align*}
We finish by taking $\delta<\min\!\big\{\frac{\epsilon}{1+\norm{\xi_0}_{\mathcal E}},\frac{\epsilon}{2}\big\}$\,. 
\end{proof}

\begin{prop}\label{A3}
The discretization $\Gamma^{\rm dis}$ is an algebraically irreducible representation. 
\end{prop}

\begin{proof}
We need to show that every non-null vector $\xi_0$ is cyclic. This can be restated as the surjectivity of the operator 
$$
T_{\xi_0}:\ell^1\big(\G\,\vert\,\mathscr C\big)\to\mathcal E,\quad T_{\xi_0}(\varphi):=\Gamma^{\rm dis}(\varphi)\xi_0\,,
$$ 
which is a consequence of Lemmas \ref{A2} and \ref{ubiforti}. 
\end{proof}

%??????????????????????????????????????????
\subsection{Proof of Theorem \ref{cruciala}}\label{proof}
%???????????????????????????????????????????

\begin{defn}\label{stricharz}
Let $\Gamma:\mathfrak B\to\mathbb B(\mathcal E)$ be a representation of the Banach $^*$-algebra $\mathfrak B$ in the Banach space $\mathcal E$\,. The representation is called {\it preunitary} if there exists a Hilbert space $\mathcal H$\,, a topologically irreducible $^*$-representation $\Pi:\mathfrak B\to\mathbb B(\mathcal H)$ and an injective linear and bounded operator $W:\mathcal E\to\mathcal H$ such that
\begin{equation*}\label{twix}
W\Gamma(\phi)=\Pi(\phi)W,\quad\forall\,\phi\in\mathfrak B\,.
\end{equation*}
\end{defn}

Our interest in this notion lies in the next characterization, taken from \cite{Le}:

\begin{lem}\label{deunde}
The Banach $^*$-algebra $\mathfrak B$ is symmetric if and only if all its non-trivial algebraically irreducible representations are preunitary.
\end{lem}

Let us now prove Theorem \ref{cruciala}.

\begin{proof} 
We start the proof with a non-trivial algebraically irreducible representation $\Gamma:L^1(\G\,\vert\,\mathscr C)\to\mathbb B(\mathcal E)$ and, as above, we denote by $\Gamma^{\rm dis}$ its discretization. By Proposition \ref{A3}, it is also algebraically irreducible. Since $\ell^1\big(\G\,\vert\,\mathscr{C}\big)$ is symmetric, Lemma \ref{deunde} says that $\Gamma^{\rm dis}$ is preunitary, hence there exists a topologically irreducible $^*$-representation $\Pi^{\rm dis}:\ell^1\big(\G\,\vert\,\mathscr{C}\big)\to\mathbb{B}(\mathcal{H})$ and an injective bounded linear operator $W:\mathcal E\to \mathcal H$ such that 
\begin{equation}\label{flori}
W\,\Gamma^{\rm dis}(\varphi)=\Pi^{\rm dis}(\varphi)W,\quad\forall\,\varphi\in\ell^1\big(\G\,\vert\,\mathscr{C}\big)\,.
\end{equation}
The representation $\Pi^{\rm dis}$ induces a representation $\pi^{\rm dis}$ of the Fell bundle $\mathscr C^{\rm dis}$ in the same Hilbert space  $\mathcal H$\,,
which also satisfies $\pi^{\rm dis}(a^*)=\pi^{\rm dis}(a)^*$ for every $a\in\mathscr C$. 
Similarly to Remark \ref{same}, it is enough to set $\pi^{\rm dis}:=\Pi^{\rm dis}\!\circ\mu$\,, where $\mu$ has been defined in \eqref{lagarta}; one easily checks that
\begin{equation*}\label{check}
\Pi^{\rm dis}(\varphi)=\sum_{y\in G}\pi^{\rm dis}\big[\varphi(y)\big]\,,\quad \forall\,\varphi\in\ell^1\big(\G\,\vert\,\mathscr{C}\big)\,.
\end{equation*}
For each $a\in\mathscr C$, using \eqref{flori} and Remark \ref{same}, we obtain 
\begin{equation}\label{comm}
\quad\pi^{\rm dis}(a)W=\Pi^{\rm dis}\big[\mu(a)\big]W=W\Gamma^{\rm dis}\big[\mu(a)\big]=W\gamma^{\rm dis}(a)\,.
\end{equation} 
%Let us verify this assertion, writing $\xi=\Gamma^{\rm dis}(\psi)\xi_0$: 
%\begin{align*} 
%\pi(a)V\xi&=\pi(a)V\Gamma^{\rm dis}(\psi)\xi_0 \\
%&=\Pi^{\rm dis}\big(a\psi(q(a)^{-1}\,\cdot\,)\big)V\xi_0 \\
%&=V\Gamma^{\rm dis}\big(a\psi(q(a)^{-1}\,\cdot\,)\big)\xi_0 \\
%&=V\gamma(a)\Gamma^{\rm dis}(\psi)\xi_0 \\
%&=V\gamma(a)\xi.
%\end{align*} 
Let us define the $^*$-representation $\Pi:L^1(\G\,\vert\,\mathscr C)\to\mathbb B(\mathcal H)$ by 
\begin{equation}\label{labar}
\Pi(\Phi)\xi:=\int_\G \pi^{\rm dis}\big(\Phi(x)\big)\xi\,\d x\,.
\end{equation} 
Recall that $\gamma^{\rm dis}=\gamma$ (see Remark \ref{details}). We have
\begin{align*}
W\Gamma(\Phi)\xi\overset{\eqref{froc}}{=}W\int_\G \!\gamma\big(\Phi(x)\big)\xi\,\d x\overset{\eqref{comm}}{=}\int_\G\pi^{\rm dis}\big(\Phi(x)\big)W\xi\,\d x\overset{\eqref{labar}}{=}\Pi(\Phi)W\xi.
\end{align*} 
This means that $\Pi(\Phi)W=W\Gamma(\Phi)$ for every $\Phi\in L^1(\G\,\vert\,\mathscr C)$\,. So the representation $\Gamma$ is preunitary, hence $L^1(\G\,\vert\,\mathscr C)$ is symmetric.
\end{proof}

%-------------------------------------------------------------------------------------------------------
\subsection{Twisted partial group actions and crossed products}\label{goam}
%-------------------------------------------------------------------------------------------------------

Let $\Theta:=(\G,\te,\o,\A)$ be {\it a continuous twisted partial action} of the locally compact group $\G$ on the $C^*$-algebra $\A$\,. We refer to \cite{Ex1,Ab} for an exposure of the general theory. In any case, the action is implemented by isomorphisms between ideals
\begin{equation*}\label{isoideal}
\te_x:\A_{x^{-1}}\!\to\A_x\,,\quad x\in\G
\end{equation*}
and unitary multipliers
\begin{equation*}\label{medicament}
\o(y,z)\in\mathbb{UM}\big(\A_y\cap\A_{yz}\big)\,,\quad y,z\in\G
\end{equation*}
satisfying suitable algebraic and topological axioms. (See \cite[Def. 2.1]{Ex1})

\smallskip
In \cite{Ex1}, Exel associates to $(\G,\te,\o,\A)$ {\it the twisted partial semidirect product Fell bundle} in the following way: The total space is 
\begin{equation*}\label{zumab}
\mathscr C(\Theta):=\big\{(a,x)\,\big\vert\,a\in\A_x\big\}
\end{equation*}
with the topology inherited from $\A\!\times\!\G$ and the obvious bundle projection $p(a,x):=x$\,. One has 
\begin{equation}\label{racetamol}
\mathfrak C_x(\Te)=\A_x\!\times\!\{x\}\,,\quad \forall\,x\in\G\,,
\end{equation}
with the Banach space structure transported from $\A_x$ in the obvious way. The algebraic laws of the bundle are
\begin{equation}\label{product}
(a,x)\bu_\Te(b,y):=\big(\te_x\big[\te_x^{-1}(a)b\big]\o(x,y),xy\big)\,,\quad\forall\,x,y\in\G\,,\,a\in\A_x\,,\,b\in\A_y\,,
\end{equation}
\begin{equation}\label{invol}
(a,x)^{\bu_\Te}\!:=\big(\te_x^{-1}(a^*)\o(x^{-1},x)^*\!,x^{-1}\big)\,,\quad\forall\,x\in\G\,,\,a\in\A_x\,.
\end{equation}

\begin{rem}\label{spanak}
In \cite{Ex1} it is shown that twisted partial semidirect product Fell bundles are very general: every separable Fell bundle with stable unit fiber $\mathfrak C_\e$ is of such a type.
\end{rem}

For a (continuous) section $\Phi$ and for any $x\in\G$ one has $\Phi(x)=\big(\tilde\Phi(x),x\big)$\,, with $\tilde\Phi(x)\in\A_x\subset\A$\,. This allows identifying $\Phi$ with a function $\tilde\Phi:\G\to\A$ such that $\tilde\Phi(x)\in\A_x\subset\A$ for each $x\in\G$\,. By this identification $L^1\big(\G\,\vert\,\mathscr C(\Te)\big)$ can be seen as $L_{\Te}^1\big(\G,\A\big)$\,, the completion of
\begin{equation*}\label{pisici}
C_{\rm c}\big(\G,\{\A_x\}_x\big):=\big\{\tilde\Phi\in C_{\rm c}(\G,\A)\,\big\vert\,\tilde\Phi(x)\in\A_x\,,\forall\,x\in\G\big\}
\end{equation*}
in the norm
\begin{equation*}\label{flax}
\p\!\tilde\Phi\!\p_{L^1_\Te(\G,\A)}\,:=\int_\G\!\p\tilde\Phi(x)\!\p_\A \!dx\equiv\,\p\!\tilde\Phi\!\p_{L^1(\G,\A)}\,,
\end{equation*}
so it sits as a closed (Banach) subspace of $L^1(\G,\A)$. The substantial difference is the semidirect product Fell bundle algebraic structure of $C_{\rm c}\big(\G,\{\A_x\}_x\big)$\,, consequence of \eqref{broduct},\,\eqref{inwol},\,\eqref{product} and \eqref{invol}\,:
$$
\big(\tilde\Phi*_{\Te}\!\tilde\Psi\big)(x)=\int_\G \te_y\big[\te_y^{-1}\big(\tilde\Phi(y)\big)\tilde\Psi(y^{-1}x)\big]\o\big(y,y^{-1}x\big)\d y
$$ 
and
$$
\tilde\Phi^{*_\Te}(x)=\Delta(x^{-1})\te_{x}\big[\tilde\Phi(x^{-1})^*\big]\o(x^{-1}\!,x)\,,
$$
which extends to $L^1_\Te(\G,\A)$\,. The $C^*$-envelope of $L^1_\Te(\G,\A)$ is $\A\rtt_\Te\!\G:=C^*\big(\G\,\vert\,\mathscr C(\Te)\big)$ and it is called {\it the (partial, twisted) crossed product of $\,\G$ and $\A$}\,. The translation of Theorems \ref{theoremix} and \ref{theoremix1} to this setting implies the following corollary.

\begin{cor}\label{fetitele}
Let $\Theta:=(\G,\te,\o,\A)$ be a continuous twisted partial action of the locally compact group $\G\,$ for which ether the discrete group $\G^{\rm dis}$ is rigidly symmetric, or it is symmetric and $\A\rtt_\Te\G^{\rm dis}$ is type $I$\,. Then $L^1_\Te(\G,\A)$ is a symmetric Banach $^*$-algebra.
\end{cor}

\begin{rem}\label{Are}
A (very) particular case it the one of {\it  global} twisted actions, basically characterized by $\A_x=\A$ for every $x\in\G$\,. Even more particularly, one may take $\A=\mathbb C$\,, with the trivial action, and then $\o:\G\!\times\!\G\to\T$ is {\it a multiplier}. In \cite{Au} Austad shows for such a case similar results, but assuming that the extension $\G_\o$ of $\T$ by $\G$ associated to $\o$ is symmetric. These are good assumptions, since rigid symmetry or discretization are not required. On the other hand, the extension could be much more complicated than the group itself.
\end{rem}

%-------------------------------------------------------------------------------------------------------
\subsection{A weighted symmetry result}\label{flocinyor}
%-------------------------------------------------------------------------------------------------------

A {\it weight} on the locally compact group $\G$ is a continuous function $\nu: \G\to [1,\infty)$ satisfying 
\begin{equation*}\label{submultiplicative}
 \nu(xy)\leq \nu(x)\nu(y)\,,\quad\nu(x^{-1})=\nu(x)\,,\quad\forall\,x,y\in\G\,.
\end{equation*}
This gives rise to the Banach $^*$-algebra $L^{1,\nu}(\G)$ defined by the norm
\begin{equation*}\label{ponderata}
\p\!\psi\!\p_{L^{1,\nu}(\G)}\,:=\int_\G\nu(x)|\psi(x)|dx\,.
\end{equation*}
Let $\mathscr C$ be a Fell bundle over $\G$\,. On $\,C_{\rm c}(\G\!\mid\!\mathscr C)$  we introduce the norm
\begin{equation}\label{normix}
\p\!\Phi\!\p_{L^{1,\nu}(\G\mid\mathscr C)}\,\equiv\,\p\!\Phi\!\p_{L^{1,\nu}}\,:=\int_\G\nu(x)\!\p\!\Phi(x)\!\p_{\mathfrak C_x}\!dx\,.
\end{equation}
The completion in this norm, denoted by $L^{1,\nu}(\G\!\mid\!\mathscr C)$\,, is a Banach $^*$-algebra with the algebraic structure inherited from $L^1(\G\!\mid\!\mathscr C)$\,. 
In this context we are going to improve Theorem \ref{theoremix}.

\begin{thm}\label{oix}
Let $\G$ be a locally compact group which is rigidly symmetric as a discrete group and $\nu$ a weight. Assume that there exists a generating subset $U$ of $\,\G$ containing the unit $\e$ such that almost everywhere
\begin{enumerate}
\item[(a)] 
\begin{equation}\label{sprot}
\lim_{n\rightarrow \infty}\sup_{x \in U^n}\nu(x)^{\frac{1}{n}}=1\,,
\end{equation}
\item[(b)] 
for some finite constant $C$ one has for any $n\in\mathbb{N}$
\begin{equation}\label{limmit}
\sup_{x\in U^n\setminus U^{n-1}}\!\nu(x) \leq C\!\inf_{x\in U^n\setminus U^{n-1}}\!\nu(x)\,.
\end{equation}
\end{enumerate}
Then $L^{1,\nu}(\G\!\mid\!\mathscr C)$ is a symmetric Banach $^*$-algebra for every Fell bundle $\mathscr C$ over $\G$\,.
\end{thm}

\begin{proof}
We are going to adapt to our more general setting the strategy from \cite{FGL}, involving spectral radii and Hulanicki's Lemma (in the form of \cite[Lemma\,3.1]{FGL0}, to allow for non-unital algebras); see also \cite{JM}. 

\smallskip
As $\nu(\cdot)\geq 1$\,, we have $\p\!\cdot\!\p_{L^1}\,\leq\,\p\!\cdot\!\p_{L^{1,\nu}}$\,;
the spectral radius formula implies that 
$$
r_{L^{1}}(\Phi)\leq r_{L^{1,\nu}}(\Phi)\,,\quad\forall\,\Phi\in {L^{1,\nu}(\G\!\mid\!\mathscr{C})}\,.
$$
To prove the opposite inequality, pick $U$ as in the statement. Later on we will also need the product of several elements. From \eqref{broduct} we get 
\begin{equation}\label{mutzi}
\begin{aligned}
&\big(\Phi_1\ast\dots\ast\Phi_n\big)(x)\\
&=\int_\G\int_\G\dots\int_\G \Phi_1(y_1)\bu\Phi_2\big(y_1^{-1}y_2\big)\bu\dots\bu\Phi_n\big(y_{n-1}^{-1}x\big)\,dy_1 dy_2\dots dy_{n-1}\,.
\end{aligned}
\end{equation}Setting $\Phi_1=\dots=\Phi_n\equiv\Phi$ in \eqref{mutzi}, combining this with \eqref{normix}, making a change of variables and using the decomposition $\G=\bigsqcup_{m\in\mathbb N}(U^m\setminus U^{m-1})$\,, one easily obtains
\begin{equation}\label{forteen}
\begin{aligned}
 &\quad\quad\p\!\Phi^{\ast n} \!\p_{L^{1,\nu}}\\
 &\leq \int_\G\!\dots\!\int_\G\!\nu(x)\!\p\!\Phi(y_1) \!\p_{\mathfrak C_{y_1}}\!\p\!\Phi\big(y_1^{-1}y_2\big) \!\p_{\mathfrak C_{y_1^{-1}y_2}}\!\!\!\dots \!\p\!\Phi(y_{n-1}^{-1}x) \!\p_{\mathfrak C_{y_n^{-1}x}}\!\!dy_1\dots dy_{n-1}dx\\
 &= \int_\G\cdots\int_\G\nu(x_1\dots x_n)\!\p\!\Phi(x_1) \!\p_{\mathfrak C_{x_1}}\cdots \!\p\!\Phi(x_n) \!\p_{\mathfrak C_{x_n}}\!dx_1\dots dx_n\\
&=\!\sum_{m_1,\dots,m_n}\underset{U^{m_1}\setminus U^{m_1-1}}{\int}\!\!\!\dots\!\!\!\underset{U^{m_n}\setminus U^{m_n-1}}{\int}\!\nu(x_1\dots x_n)\!\p\!\Phi(x_1) \!\p_{\mathfrak C_{x_1}}\!\dots \!\p\!\Phi\big(x_n) \!\p_{\mathfrak C_{x_n}} \!dx_1\dots dx_n\,.
\end{aligned}
\end{equation}
Define
\begin{equation*}
 v(n):=\sup_{x\in U^{|n|}} \nu(x)\,,\quad\forall\,n\in\Z\,.
\end{equation*}
It is a weight and one has the obvious associated weighted space $\ell^{1,\nu}(\Z)$\,. 
If $x_j\in U^{m_j}\setminus U^{m_j-1}$, then $x_1\cdots x_n \in U^{m_1+\cdots +m_n}$ and so the weight satisfies
\begin{equation*}
\nu(x_1\cdots x_n)\leq \sup_{y\in U^{m_1+\cdots +m_n}} \nu(y) =v(m_1+\cdots m_n)\,.
\end{equation*}
Set $\beta(m):=\int_{U^m\setminus U^{m-1}}\!\!\p\!\Phi(x)\!\p_{\mathscr C_x}\!dx$ and extend it by null values for negative $m$\,. Then we have $\p\!\Phi \!\p_{L^1}\,=\,\p\beta\!\p_{\ell^1(\mathbb Z)}$\,. The condition \eqref{limmit} implies immediately that 
\begin{equation*}
C^{-1}\!\p\!\beta\!\p_{\ell^{1,v}(\mathbb Z)}\,\leq \,\p\!\Phi \!\p_{L^{1,\nu}}\,\leq\,\p\!\beta \!\p_{\ell^{1,v}(\mathbb Z)}.
\end{equation*}
Returning to \eqref{forteen} and denoting by $\beta^{\star n}$ the iterated self-convolution of $\beta$\,, we obtain
\begin{equation*}
\p\!\Phi^{\ast n} \!\p_{L^{1,\nu}}\,\leq\!\sum_{m_1,\dots,m_n=1}^{\infty}\!v(m_1+\cdots +m_n)\,\beta(m_1)\cdots \beta(m_n)=\;\p\!\beta^{\star n} \!\p_{\ell^{1,v}(\mathbb Z)}\,<\infty\,.
\end{equation*}
By its definition and by \eqref{sprot}, the weight $v$ on the group $\mathbb Z$ satisfies the GRS-condition $\lim_{n\rightarrow \infty}v(n)^{\frac{1}{n}}=1$\,, and $\ell^{1,v}(\mathbb Z)$ is symmetric by \cite{FGL0}. Hence
\begin{align*}
r_{L^{1,\nu}}(\Phi)&=\lim_{n\rightarrow\infty} \!\p\!\Phi^{\ast n} \!\p_{L^{1,\nu}}^{1/n}\,\leq \lim_{n\rightarrow\infty} \!\p \beta^{\star n} \!\p_{\ell^{1,v}(\mathbb{Z})}^{1/n}\\
&=r_{\ell^{1,v}(\mathbb{Z})}(\beta)=r_{\ell^{1}(\mathbb{Z})}(\beta)\\
&\le\,\p\!\beta \!\p_{\ell^{1}(\mathbb{Z})}\,=\,\p\!\Phi \!\p_{L^{1}}.
\end{align*}
So for all $k\in\mathbb N$ we have
\begin{equation*}
r_{L^{1,\nu}}(\Phi)=r_{L^{1,\nu}}\big(\Phi^{\ast k}\big)^{1/k}\!\leq\,\p\!\Phi^{\ast k} \!\p_{L^{1}}^{1/k}.
\end{equation*}
Letting $k\rightarrow\infty$ we obtain the required inequality $\,r_{L^{1,\nu}}(\Phi)\leq r_{L^{1}}(\Phi)$\,, which finishes the proof by Hulanicki's Lemma, since $L^1(\G\,\vert\,\mathscr C)$ is symmetric, by Theorem \ref{theoremix}.
\end{proof}

\begin{rem}\label{banachalgebraic}
The algebra $L^{1,\nu}(\G\,\vert\,\mathscr C)$ is very close in nature to the algebras previously considered. One can form a Banach $^*$-algebraic bundle $\mathscr C^\nu$ (it fails to satisfy the $C^*$-condition) such that 
$$
L^{1,\nu}(\G\,\vert\,\mathscr C)=L^1(\G\,\vert\,\mathscr C^\nu)\,.
$$ The bundle $\mathscr C^\nu$ is easily constructed by taking $\mathscr C$ as the underlying space but changing the norm in $\mathfrak C_x$ to the equivalent one $\norm{\cdot}_{\mathfrak C_x^\nu}:=\nu(x)\norm{\cdot}_{\mathfrak C_x}$\,.
\end{rem}

%??????????????????????????????????????????!!!!!!!!!!!!!!!!!!!!!!!!!!!!!!!
\subsection{Kernels for Fell bundles over locally compact groups}\label{pishpish}
%???????????????????????????????????????????!!!!!!!!!!!!!!!!!!!!!!!!!!!! 

We are now interested in vector-valued kernels in the presence of a Fell bundle $\mathscr C\overset{q}{\to}\G$ and a weight $\nu$\,. We assume that $\G$ is unimodular.

\begin{defn}\label{nucleo}
{\it A kernel-section} is a measurable function $K:\G^2\equiv\G\!\times\!\G\to\mathscr C$ such that $K(x,y)\in\mathfrak C_{xy^{-1}}$ for almost every $x,y\in\G$\,.  It is {\it $\nu$-convolution-dominated}, and we write $K\in\mathfrak K^\nu\big(\G^2\vert\,\mathscr C\big)$\,, if there exists $k\in L^{1,\nu}(\G)$ such that almost everywhere
\begin{equation}\label{camioneta}
\norm{K(x,y)}_{\mathfrak C_{xy^{-1}}}\!\le k\big(xy^{-1}\big)\,.
\end{equation}
Such a kernel-section is called {\it covariant} if $K(xz,yz)=K(x,y)$ almost everywhere. We write $K\in\mathfrak K_{\rm cov}^\nu\big(\G^2\vert\,\mathscr C\big)$\,.
When $\nu=1$ we write simply $\mathfrak K\big(\G^2\vert\,\mathscr C\big)$ and $\mathfrak K_{\rm cov}\big(\G^2\vert\,\mathscr C\big)$\,.
\end{defn}

It is not difficult to verify that $\mathfrak K^\nu\big(\G^2\vert\,\mathscr C\big)$ is a Banach $^*$-algebra with the product
\begin{equation*}\label{brod}
\big(K\diamond L\big)(x,y):=\int_\G K(x,z)\bu L(z,y)\,\d z\,,
\end{equation*}
the involution
\begin{equation*}\label{inw}
K^{\diamond}(x,y):=K(y,x)^\bu\,
\end{equation*}
and the norm
\begin{equation*}\label{orm}
\left\Vert K \right\Vert_{\mathfrak K^\nu} := \inf\left\{ \Vert k \Vert_{L^{1,\nu}(\G)}\,\big\vert\,k \textrm{ satisfies \eqref{camioneta}}\right\}.
\end{equation*}
In addition, $\mathfrak K_{\rm cov}^\nu\big(\G^2\vert\,\mathscr C\big)$ is a closed $^*$-subalgebra.

\begin{prop}\label{inchep}
The mapping $\Gamma:L^{1,\nu}(\G|\mathscr{C})\rightarrow \mathfrak K^\nu\big(\G^2\,\vert\,\mathscr C\big)$ given by 
\begin{equation*}\label{incap}
(\Gamma\Phi)(x,y):=\Phi\big(xy^{-1}\big)\,,\,\,\,\,\forall\,x,y\in\G\,,
\end{equation*}
is an isometric $^*$-algebraic morphism. Its range is $\mathfrak K^\nu_{\rm cov}\big(\G^2\vert\,\mathscr C\big)$\,. On this range, the inverse reads
\begin{equation*}\label{incep}
\big(\Gamma^{-1}K\big)(x):=K(x,\e)\,,\quad \forall\,x\in\G\,.
\end{equation*}
\end{prop}

\begin{proof}
Obviously, the values of $\Gamma$ are all kernel-sections and are covariant. Each $\Gamma\Phi$ is convolution dominated, with $k(x)\!:=\,\p\!\Phi(x)\!\p_{\mathfrak C_x}$\,, and
\begin{equation*}\label{trompetica}
\p\!\Gamma\Phi\!\p_{_{\mathfrak K^\nu}}\,=\,\p\!k\!\p_{L^{1,\nu}(\G)}\,\overset{\eqref{normix}}{=}\,\p\!\Phi\!\p_{L^{1,\nu}(\G\mid\mathscr C)}\,.
\end{equation*}

All the algebraic verifications are trivial.
\end{proof}

\begin{cor}\label{increp}
In the setting of Theorem \ref{oix}, for every Fell bundle $\mathscr C$ over $\G$\,, the Banach $^*$-algebra $\mathfrak K_{\rm cov}^\nu\big(\G^2\,\vert\,\mathscr C\big)$ is symmetric.
\end{cor}

\begin{proof}
This follows from Proposition \ref{inchep} and Theorem \ref{oix}.
\end{proof}

To find a larger symmetric Banach $^*$-algebra of convolution-dominated section-kernels, from the Fell bundle $\mathscr C\overset{q}{\to}\G$ with algebraic operations $\big(\bu,^\bu\big)$ we construct a new Fell bundle $\tilde{\mathscr C}\overset{Q}{\to}\G$\,, with fibres
\begin{equation*}\label{jepurash}
\begin{aligned}
&\tilde{\mathfrak C}_x\equiv{\rm RUC}\big(\G,\mathfrak C_x\big)\\
&:=\big\{f:\G\to\mathfrak C_x\,\big\vert\,f\ {\rm bounded\ and\ continuous}\,,\ \lim_{z\to \e}\sup_{y}\norm{f(z^{-1}y)-f(y)}_{\mathfrak C_x}\!=0\big\}\,,
\end{aligned}
\end{equation*} 
norm
\begin{equation*}\label{cincimin}
\norm{f}_{\tilde{\mathfrak C}_x}\!:=\sup_{y}\norm{f(y)}_{\mathfrak C_x}
\end{equation*}
and the algebraic structure given by
\begin{equation}\label{cuc}
(f\tilde\bu g)(z):=f(z)\bu g\big(Q(f)^{-1}z\big)\,,
\end{equation}
\begin{equation}\label{singur}
f^{\tilde\bu}(z):=f\big(Q(f)z\big)^\bu.
\end{equation}
It is easy to check all the axioms. Using notations of the form 
\begin{equation*}\label{framanta}
[F(x)](z)\equiv F(x,z)\in\mathfrak C_x\,,\quad x,z\in\G\,,
\end{equation*}
the algebraic operations on $L^{1,\nu}\big(\G\!\mid\!\tilde{\mathscr C}\,\big)$ are
\begin{equation*}\label{velmonte}
\big(F\tilde\ast G\big)(x,z)=\int_\G F(y,z)\bu G\big(y^{-1}x,y^{-1}z\big)\d y\,,
\end{equation*}
\begin{equation*}\label{aquilar}
F^{\tilde\ast}\!(x,z):=F\big(x^{-1}\!,x^{-1}z\big)^\bu.
\end{equation*}

\begin{rem}\label{jarp}
Note that $\mathscr C$ can be seen as a Fell sub-bundle of $\tilde{\mathscr C}$\,; in each fiber one only consider the constant functions. Consequently $L^{1,\nu}\big(\G\!\mid\!\mathscr C\big)$ may be identified with a $^*$-subalgebra of $L^{1,\nu}\big(\G\!\mid\!\tilde{\mathscr C}\,\big)$\,. 
\end{rem}

\begin{ex}\label{bucluc}
Let $(\A,\alpha,\G)$ be a global $C^*$-algebraic dynamical system \cite{Wil} (the partial version is more complicated but similar). For simplicity $\A$ is unital. From it we construct the Fell bundle with $\mathscr C_x:=\A\!\times\!\{x\}$ for every $x\in\G$ and with algebraic rules
\begin{equation*}\label{murtzi}
(a,x)\bu(b,y):=\big(a\alpha_x(b),xy\big)\,,\quad(a,x)^\bu:=\big(\alpha_x^{-1}(a)^*,x^{-1}\big)\,.
\end{equation*}
One has $(a,x)=(a,\e)\bu(1_\A,x)$\,.
The reader can check easily that, modulo suitable identifications, the Fell bundle $\tilde{\mathscr C}$ arises from the dynamical system $\big({\rm RUC}(\G,\A),\tilde\alpha,\G\big)$ for the global action
\begin{equation*}\label{identif}
\tilde\alpha:\G\to{\rm Aut}\big[{\rm RUC}(\G,\A)\big]\,,\quad\big[\tilde\alpha_x(f)\big](z):=\alpha_x\big[f(x^{-1}z)\big]\,.
\end{equation*}
There are also cohomologically twisted versions of the global and even of the partial dynamical system \cite{Ex1} that fit in the Fell bundle framework, that we do not make explicit here.
\end{ex}

\begin{prop}\label{suculent}
The map $\Si:L^{1,\nu}\big(\G\,\vert\,\tilde{\mathscr C}\,\big)\to\mathfrak K^\nu\big(\G^2\,\big\vert\,\mathscr C)$\,, given by
\begin{equation}\label{patt}
\big[\Si(F)\big](x,y):=\big[F\big(xy^{-1}\big)\big](x)\equiv F\big(xy^{-1}\!,x\big)\,,
\end{equation} 
%\begin{equation}\label{tap}
%\big[\Gamma^{-1}(K)\big](x,y):=K\big(y,x^{-1}y\big)\,,
%\end{equation} 
is an isometric morphism of Banach $^*$-algebras. Its range, denoted by $\mathfrak K^\nu_{\rm RUC}\big(\G^2\,\big\vert\,\mathscr C)$\,, can be characterized as the family of elements $K\in\mathfrak K^\nu\big(\G^2\,\big\vert\,\mathscr C)$ such that for almost every $z\in\G$ 
\begin{equation*}\label{lash}
 \lim_{y\to \e}{\rm ess}\sup_{x}\norm{K\big(y^{-1}x,z^{-1}y^{-1}x\big)-K\big(x,z^{-1}x\big)}_{\mathfrak C_x}\!=0\,.
\end{equation*}
\end{prop}

\begin{proof}
Although not strictly necessary, a fuller story is told in the next commutative diagram:
$$
\begin{diagram}
\node{L^{1,\nu}\big(\G\,\vert\,\mathscr C\,\big)}\arrow{s,l}{\Gamma}\arrow{e,t}{}\node{L^{1,\nu}\big(\G\,\vert\,\tilde{\mathscr C}\,\big)}\arrow{se,t}{\Si}\arrow{s,l}{\tilde{\Gamma}}\\ 
\node{\mathscr K^\nu_{\rm cov}\big(\G^2\,\vert\,\mathscr C\big)}\arrow{n}{}\arrow{e,t}{}\node{\mathscr K_{\rm cov}^\nu\big(\G^2\,\vert\,\tilde{\mathscr C}\,\big)}\arrow{n}{}\arrow{e}\arrow{e,t}{\Omega}\node{\mathscr K^\nu\big(\G^2\,\vert\,\mathscr C\big)}
\end{diagram}
$$
The unlabeled horizontal arrows are just obvious inclusions; see Remark \ref{jarp}. Proposition \ref{inchep} applied to the Fell bundle $\tilde{\mathscr C}$ justifies the isomorphism $\tilde\Gamma$. The morphism $\Omega:\mathscr K_{\rm cov}^\nu\big(\G^2\,\vert\,\tilde{\mathscr C}\,\big)\to \mathscr K^\nu\big(\G^2\,\vert\,\mathscr C\big)$ is given by
\begin{equation*}\label{diare}
(\O\mathbf K)(x,y):=\big[\mathbf K(x,y)\big](x)\equiv\mathbf K(x,y;x)\,,
\end{equation*}
so that composing $\O$ with $\tilde\Gamma$ we get \eqref{patt}. Let us check that $\O$ is indeed a morphism. For the product, using natural notations, one gets:
$$
\begin{aligned}
\O(\mathbf K\tilde\diamond\mathbf L)(x,y)&=(\mathbf K\tilde\diamond\mathbf L)(x,y;x)\\
&=\int_\G\big[\mathbf K(x,z)\tilde\bu\mathbf L(z,y)\big](x)dz\\
&\overset{\eqref{cuc}}{=}\int_\G\mathbf K(x,z;x)\bu\mathbf L(z,y;zx^{-1}x)dz\\
&=\int_\G(\O\mathbf K)(x,z)\bu(\O\mathbf L)(z,y)dz\\
&=(\O\mathbf K\diamond\O\mathbf L)(x,y)\,.
\end{aligned}
$$
When aplying \eqref{cuc}, we also used the fact that $Q\big[\mathbf K(x,z)\big]^{-1}\!=zx^{-1}$. For the involution one can write
$$
\begin{aligned}
\big(\O\mathbf K^{\tilde\diamond}\big)(x,y)&=\big[\mathbf K^{\tilde\diamond}(x,y)\big](x)\\
&=\big[\mathbf K(y,x)\big]^{\tilde\bu}(x)\\
%&=\big[\mathbf K(y,x)\big]^{\tilde\bu}(x)\\
&\overset{\eqref{singur}}{=}\big[\mathbf K(y,x)\big]\big(yx^{-1}x\big)^{\bu}\\
&=(\O\mathbf K)(y,x)^\bu\\
&=(\O\mathbf K)^\diamond(x,y)\,.
\end{aligned}
$$
For isometry, it is sufficient (and easier) to work with $\Si$\,, first defined on $C_{\rm c}\big(\G\,\vert\,\tilde{\mathscr C}\,\big)$\,. In this case, the condition $\p\!(\Si F)(x,y)\!\p_{\mathfrak C_{xy^{-1}}}\le k\big(xy^{-1}\big)$ for every $x,y$ can be rewritten as 
$$
\p\!F(z)\!\p_{\tilde{\mathfrak C}_{z}}=\sup_x\!\p\!F(z,x)\!\p_{\mathfrak C_{z}}\,\le k(z)\,,\quad\forall\,z\in\G\,.
$$ 
Then one gets the isometry just by applying the definitions of the norms.
The characterization of the range follows from an examination of the mapping $\Si$\,.
\end{proof}

\begin{cor}\label{increpp}
In the setting of Theorem \ref{oix}, for every Fell bundle $\mathscr C$ over $\G$\,, the Banach $^*$-algebra $\mathfrak K_{\rm RUC}^\nu\big(\G^2\,\vert\,\mathscr C\big)$ is symmetric.
\end{cor}

\begin{proof}
Consequence of Proposition \ref{suculent} and Theorem \ref{oix}.
\end{proof}

\section*{Acknowledgements}

F. Flores has been partially supported by the Fondecyt Project 1171854 and the Fondecyt Project 1200884. D. Jaur\'e acknowledges financial support from Beca de Doctorado Nacional Conicyt. M. M\u antoiu has been supported by the Fondecyt Project 1200884. We are grateful to V. Nistor; due to his advice, the paper improved a lot.

%-------------------------------------------------------------------------------------------------------

ADDRESS 

\medskip
F. Flores:

\smallskip
Departamento de Ingenier\'ia Matem\'atica, Universidad de Chile, 

Beauchef 851, Torre Norte, Oficina 436,

Santiago, Chile.

E-mail: feflores@dim.uchile.cl

\medskip
D. Jaur\'e:

\smallskip
Facultad de Ciencias, Departamento de Matem\'aticas, Universidad de Chile

Las Palmeras 3425, 

Santiago, Chile.

E-mail: diegojaure@ug.uchile.cl

\medskip
M. M\u antoiu:

\smallskip
Facultad de Ciencias, Departamento de Matem\'aticas, Universidad de Chile

Las Palmeras 3425, Casilla 653, 

Santiago, Chile.

E-mail: mantoiu@uchile.cl

\end{document}